\documentclass[12pt]{amsart}

\usepackage{amsmath,amssymb,amsthm,amscd,enumerate,setspace}
\usepackage[all]{xy}
\usepackage{xcolor}
\usepackage{multirow}

\usepackage[pagebackref]{hyperref}
\hypersetup{colorlinks=true, linkcolor=red, citecolor=blue}

\long\def\symbolfootnote[#1]#2{\begingroup%
\def\thefootnote{\fnsymbol{footnote}}\footnote[#1]{#2}\endgroup}

\newtheorem{Theorem}{Theorem}[section]
\newtheorem{Lemma}[Theorem]{Lemma}
\newtheorem{Corollary}[Theorem]{Corollary}
\newtheorem{Proposition}[Theorem]{Proposition}
\newtheorem{claim}{Claim}

\theoremstyle{definition}

\newtheorem{Remark}[Theorem]{Remark}
\newtheorem{Example}[Theorem]{Example}

\numberwithin{equation}{Theorem}

\begin{document}

\title{Nearly Gorenstein local rings defined by maximal minors of a $2 \times n$ matrix}

\author[S. Kumashiro]{Shinya Kumashiro}
\address{National Institute of Technology, Oyama College, 771 Nakakuki, Oyama, Tochigi, 323-0806, Japan}
\email{skumashiro@oyama-ct.ac.jp}

\author[N. Matsuoka]{Naoyuki Matsuoka}
\address{Department of Mathematics, School of Science and Technology, Meiji University, 1-1-1 Higashi-mita, Tama-ku, Kawasaki 214-8571, Japan}
\email{naomatsu@meiji.ac.jp}

\author[T. Nakashima]{Taiga Nakashima}
\address{Department of Mathematics, School of Science and Technology, Meiji University, 1-1-1 Higashi-mita, Tama-ku, Kawasaki 214-8571, Japan}
\email{tiger07dahe@gmail.com}

\thanks{2020 {\em Mathematics Subject Classification.} 13H10, 13A02, 13A15.}
\keywords{Nearly Gorenstein rings, Almost Gorenstein rings, Canonical ideals, Trace ideals, Minimal free resolutions, Numerical semigroup rings}
\thanks{Shinya Kumashiro was supported by JSPS KAKENHI Grant Number JP21K13766 and by Grant for Basic Science Research Projects from the Sumitomo Foundation (Grant number 2200259).}

\newcommand{\fka}{\mathfrak{a}}
\newcommand{\fkn}{\mathfrak{n}}
\newcommand{\fkm}{\mathfrak{m}}
\newcommand{\fkp}{\mathfrak{p}}
\newcommand{\fkM}{\mathfrak{M}}

\newcommand{\mult}{\operatorname{e}}
\newcommand{\Mat}{\operatorname{Mat}}
\newcommand{\trace}{\operatorname{tr}}
\newcommand{\canon}{\operatorname{K}}
\newcommand{\detid}{\operatorname{I}}
\newcommand{\rank}{\operatorname{rank}}

\newcommand{\Hom}{\operatorname{Hom}}
\newcommand{\Ker}{\operatorname{Ker}}
\newcommand{\image}{\operatorname{Im}}
\newcommand{\PF}{\operatorname{PF}}

\newcommand{\Spec}{\operatorname{Spec}}
\newcommand{\Ass}{\operatorname{Ass}}

\newcommand{\nint}[1]{[#1]}

\maketitle

\setlength{\baselineskip} {15pt}

\begin{abstract}
We investigate the nearly Gorenstein property of a local ring defined by the maximal minors of a specific $2 \times n$ matrix with entries in the formal power series ring $k[[X_1, X_2, \ldots , X_n]]$ over a field $k$. Our findings allow us to present numerous concrete examples, such as nearly Gorenstein rings that are not almost Gorenstein and vice versa.
\end{abstract}

\section{Introduction}

Let $R$ be a commutative ring and $M$ an $R$-module. The {\it trace ideal} of $M$ is defined as
$$
\trace_R(M) = \sum_{f \in \Hom_R(M,R)} \image f.
$$
Several recent studies on trace ideals have been conducted in various situations (\cite{DKT, DL, GIK, HHS, HKS, KT, L, LP}).
In particular, Herzog, Hibi, and Stamate \cite{HHS} introduced a new class of Cohen-Macaulay local rings, called {\it nearly Gorenstein} rings, through profound research on the trace ideal of the canonical module.
The Cohen-Macaulay and Gorenstein properties are fundamental in commutative ring theory, with Gorenstein rings being a special type of Cohen-Macaulay rings.
Due to the rarity of Gorenstein rings among Cohen-Macaulay rings, there is a need for intermediate classes that bridge this significant gap.
The notion of nearly Gorenstein rings is one such candidate.
Another generalization of Gorenstein rings has been developed in, for example, \cite{BF, CGKM, GMP, GTT}, known as {\it almost Gorenstein} rings.
According to \cite[Proposition 6.1]{HHS}, every one-dimensional almost Gorenstein local ring is also a nearly Gorenstein ring.
In the present article, we first explore numerical semigroup rings, where the defining ideal is generated by the maximal minors of a $2 \times n$ matrix, and we examine the subtle distinctions between nearly Gorenstein and almost Gorenstein properties.
Next, we consider a higher-dimensional analog of such rings.

Let $R$ be a Cohen-Macaulay local ring with maximal ideal $\fkm$ and assume that $R$ possesses the canonical module $\canon_R$.
We say that $R$ is a nearly Gorenstein local ring if $\trace_R(\canon_R) \supseteq \fkm$ (\cite[Definition 2.2]{HHS}). 
Besides, $R$ is said to be an almost Gorenstein local ring if there exists an exact sequence
$$
0 \to R \to \canon_R \to C \to 0
$$
of $R$-modules such that $\mu_R(C) = \mult^0_\fkm(C)$ (\cite[Definition 3.3]{GTT}), where $\mu_R(C)$ (resp. $\mult^0_\fkm(C)$) stands for the number of a minimal system of generators of $C$ (resp. the multiplicity of $C$ with respect to $\fkm$).
It is well-known that $R$ is a Gorenstein ring if and only if $\canon_R$ is isomorphic to $R$ as an $R$-module, which is equivalent to the equality $\trace_R(\canon_R) = R$.
Consequently, every Gorenstein local ring is both nearly Gorenstein and almost Gorenstein.
Furthermore, both of these properties can be seen as attempts to measure the proximity to Gorenstein rings using the canonical module.

We now consider numerical semigroup rings. Let $n > 0$ and $a_1, a_2, \ldots, a_n > 0$ be positive integers with the greatest common divisor of $a_1, a_2, \ldots, a_n$ being $1$. 
We define $H$ as the numerical semigroup generated by $a_1, a_2, \ldots, a_n$, given by
$$
H = \left<a_1, a_2, \ldots, a_n\right> = \left\{\sum_{i=1}^n \lambda_i a_i ~\middle|~ 0 \le \lambda_1, \lambda_2, \ldots, \lambda_n \in \mathbb{Z}\right\}.
$$
Throughout this article, we assume that the system $a_1, a_2, \ldots, a_n$ of generators of $H$ is minimal.
Let $k$ be a field, and consider the $k$-algebra homomorphism
$$
\varphi_H : S=k[[X_1, X_2, \ldots, X_n]] \to V=k[[t]]
$$
defined by $\varphi_H(X_i) = t^{a_i}$ for each $1 \leq i \leq n$, where $k[[X_1, X_2, \ldots, X_n]]$ and $k[[t]]$ denote the formal power series rings over $k$. The image of $\varphi_H$ is called the {\it numerical semigroup ring} of $H$, denoted by $k[[H]]$, and the kernel of $\varphi_H$ is called the {\it defining ideal} of $k[[H]]$, denoted by $I_H$. We denote the image of $X_i$ in $k[[H]]$ by $x_i = t^{a_i}$ for each $1 \le i \le n$.

For the case of $n=3$, the structure of $I_H$ is completely described by Herzog's breakthrough result in \cite{H}. Specifically, if $k[[H]]$ is not Gorenstein, then the defining ideal takes the form
$$
I_H = \detid_2
\begin{pmatrix}
X_2^{m_2} & X_3^{m_3} & X_1^{m_1} \\
X_1^{\ell_1} & X_2^{\ell_2} & X_3^{\ell_3}
\end{pmatrix}
$$
for some positive integers $m_1, m_2, m_3, \ell_1, \ell_2, \ell_3$, where $\detid_2(\mathcal{M})$ denotes the ideal generated by $2$-minors of a matrix $\mathcal{M}$. In this case, we have $\trace_{k[[H]]}(\canon_{k[[H]]}) = (x_i^{m_i}, x_i^{\ell_i} \mid i = 1,2,3)$ according to \cite[Corollary 3.4]{HHS}, which implies that $k[[H]]$ is nearly Gorenstein if and only if $\min\{m_i, \ell_i\} = 1$ for all $i = 1,2,3$. Moreover, in the same case, $k[[H]]$ is almost Gorenstein if and only if $m_1 = m_2 = m_3 = 1$ or $\ell_1 = \ell_2 = \ell_3 = 1$ (\cite[Corollary 4.2]{GMP}, \cite[Corollary 3.3]{NNW}).
However, for the case of $n \ge 4$, the form of the defining ideals is not well-established, and therefore, we do not have an explicit statement.
Here, let us discuss numerical conditions for nearly Gorenstein and almost Gorenstein properties of $R=k[[H]]$ from a different perspective.
As is well-known that the pseudo-Frobenius numbers of the numerical semigroup $H$, which are integers not in $H$ but become elements of $H$ after adding a positive element in $H$, correspond to the degrees of the generators of $\canon_R$ as an $R$-module (see \cite{GW}), providing valuable information about $\canon_R$. In fact, $R$ is almost Gorenstein if and only if there is a kind of symmetry in the pseudo-Frobenius numbers (see \cite[Proposition 2.3]{GKMT}, \cite[Theorem 2.4]{N}, for detail).
However, for the nearly Gorenstein property, except in the case of $n=3$, no simplified criterion is known as for the almost Gorenstein property. 
This is because there is no mechanical method known for computing $\trace_R(\canon_R)$ in terms of pseudo-Frobenius numbers, even though $\trace_R(\canon_R)$ is generated by monomials in $t$.
Under these situations, to further advance the analysis for the case of $n \ge 4$, this article will focus on investigating a numerical semigroup ring with a determinantal defining ideal in the following form:
$$
I_H = \detid_2
\begin{pmatrix}
X_2^{m_2} & X_3^{m_3} & \cdots & X_n^{m_n} & X_1^{m_1} \\
X_1^{\ell_1} & X_2^{\ell_2} & \cdots & X_{n-1}^{\ell_{n-1}} & X_n^{\ell_n}	
\end{pmatrix} \qquad (*)
$$
for some positive integers $m_1, m_2, \ldots, m_n, \ell_1, \ell_2, \ldots, \ell_n$. 
This generalization of the $n=3$ case allows us further to explore the nearly Gorenstein property of numerical semigroup rings.
Such numerical semigroups and semigroup rings have been studied in \cite{GKMT, KM} to analyze the relationship between the defining ideal of $k[[H]]$ and the pseudo-Frobenius numbers of $H$.

This article presents a computational method for $\trace_R(\canon_R)$ under the assumption $(*)$ by utilizing the Eagon-Northcott resolution (\cite{EN}) of $k[[H]]$. 
Additionally, this method leads us to provide the following result.

\begin{Theorem}\label{MT}
	Let $H$ be a numerical semigroup and assume the equality $(*)$ holds for some positive integers $m_1, m_2,\ldots ,m_n, \ell_1,\ell_2,\ldots , \ell_n$.
	Then the following are equivalent.
	\begin{enumerate}[{\rm (1)}]
		\item $k[[H]]$ is nearly Gorenstein.
		\item After suitable permutation of $a_1, a_2,\ldots ,a_n$, one of the following is satisfied.
		\begin{enumerate}[{\rm (a)}]
			\item $m_i = 1$ for all $1 \le i \le n$.
			\item $m_i = 1$ for all $2 \le i \le n$ and $\ell_i = 1$ for all $1 \le i \le n-2$.
		\end{enumerate}
	\end{enumerate}
\end{Theorem}

In Section 3, we give a proof of Theorem \ref{MT}. To complete the proof, we discuss how to compute $\trace_R(\canon_R)$ by using a minimal presentation of $\canon_R$ in Section 2. 

Under our assumption, it is known that condition (a) in Theorem \ref{MT} is equivalent to the condition that $k[[H]]$ is an almost Gorenstein local ring (e.g., see \cite[Theorem 7.8]{GTT}). 
According to Theorem \ref{MT}, although the class of nearly Gorenstein rings is certainly larger than that of almost Gorenstein rings, the difference is very limited because, for any $n \ge 3$, there are only three parameters $m_1, \ell_{n-1}, \ell_n$ that can be freely chosen in the complement class.

In Section 4, we consider the higher-dimensional analog of local rings discussed in Section 3. 
Let $I, J$ be subsets of $\{1,2,\ldots, n\}$ (admitting empty set). Then, the local rings dealt with in Section 4 have the following defining ideals in the formal power series ring $k[[\{X_i\}_{1\le i \le n}, \{Y_i\}_{i \in I}, \{Z_{j}\}_{j \in J}]]$:
$$
\detid_2
\begin{pmatrix}
V_2 & V_3 & \cdots & V_1 \\
U_1 & U_2 & \cdots & U_n
\end{pmatrix}
$$
where 
$$
\begin{array}{lcl}
V_r = \begin{cases}
 	X_{r}^{m_{r}} + Y_{r} & (r \in I)\\
 	X_{r}^{m_{r}} & (\text{otherwise})
 \end{cases}
 &
 \text{ and }
 &
U_r = \begin{cases}
 	X_{r}^{\ell_{r}} + Z_r & (r \in J)\\
 	X_{r}^{\ell_{r}} & (\text{otherwise})
 \end{cases}
\end{array}
$$
for $1 \le r \le n$.
We will find that when such a ring is nearly Gorenstein, its dimension is at most $4$ for $n=3$ and at most $3$ for $n=4$. 
Moreover, we will also observe that when $n \ge 5$, the resulting rings obtained through this construction have dimension at most $2$ if they are nearly Gorenstein (Theorems \ref{newAGcase} and \ref{newnonAGcase}).

Throughout this article, for positive integers $r$ and $q$, $\Mat_{r \times q}(A)$ denotes the set of all $r \times q$ matrices with entries in a commutative ring $A$.
We regard $M = (m_{ij}) \in \Mat_{r \times q}(A)$ as the canonical $A$-linear map $M : A^{\oplus q} \to A^{\oplus r}$ defined by $A^{\oplus q} \ni \mathbf{a} \mapsto M{\cdot}\mathbf{a} \in A^{\oplus r}$ corresponding the standard basis of $A^{\oplus q}$ and $A^{\oplus r}$.

\section{Computation of the canonical trace ideals}

In this section, we discuss a computational method for determining the trace ideal of the canonical module using a representation matrix of a free resolution. 
We begin with the following proposition, essentially found by Vasconcelos \cite[Remark 3.3]{V} (see also Herzog-Hibi-Stamate \cite[Proposition 3.1]{HHS}). 
Let us include a brief proof for the sake of completeness. 

\begin{Proposition}[cf. {\cite[Proposition 3.1]{HHS}}]\label{Kumashiro}
	Let $A$ be a commutative ring and $X$ be a non-zero $A$-module.
	Assume that there exists an exact sequence
	$A^{\oplus q} \overset{M}{\longrightarrow} A^{\oplus r} \overset{\varepsilon}{\to} X \to 0$
	of $A$-modules for some $0 < q,r \in \mathbb{Z}$ and $M \in \Mat_{r \times q}(A)$. We define
	$$
	\Lambda = \left\{\mathbf{f} = \begin{pmatrix}f_1 & f_2 & \cdots & f_r \end{pmatrix} \in \Mat_{1 \times r}(A) ~\middle|~\text{$\mathbf{f}{\cdot}M = 0$ in $\Mat_{1 \times q}(A)$}\right\}.
	$$
	Then, $$\trace_A(X) = \sum_{(f_1~f_2~\cdots~f_r) \in \Lambda} (f_1, f_2,\ldots , f_r)A.$$ 
\end{Proposition}

\begin{proof}
	Let $\mathbf{f} =\begin{pmatrix}f_1 & f_2 & \cdots & f_r \end{pmatrix} \in \Lambda$.
	Since $\mathbf{f}{\cdot}M = 0$ in $\Mat_{1\times q}(A)$, the chain complex
	$$
	A^{\oplus q} \overset{M}{\longrightarrow} A^{\oplus r} \overset{\mathbf{f}}{\to} A
	$$
	of $A$-modules induces a surjection $\varphi : X \cong A^{\oplus r} / \image M \to (f_1, f_2,\ldots ,f_r)A$. 
	Hence $f_i \in \image \varphi \subseteq \trace_A(X)$ for all $1 \le i \le r$. 
	Therefore, $\trace_A(X) \supseteq \sum_{\mathbf{f} \in \Lambda} (f_1, f_2,\ldots , f_r)A$.
	
	To prove the reverse inclusion, let $x_i = \varepsilon(\mathbf{e}_i)$, where $\{\mathbf{e}_i\}_{1 \le i \le r}$ is the standard basis of  $A^{\oplus r}$.
	Let $\varphi \in \Hom_A(X,A)$ and put $f_i = \varphi(x_i)$ for each $1 \le i \le r$.
	Considering the following commutative diagram
	$$
	\xymatrix{
	A^{\oplus q} \ar[r]^M & A^{\oplus r} \ar[r]^\varepsilon \ar[rd] & X \ar[d]^{\varphi} \ar[r] & 0\\
	&& (f_1,f_2, \ldots ,f_r)A&
	}
	$$
	we have a chain complex $A^{\oplus q} \overset{M}{\longrightarrow} A^{\oplus r} \overset{\mathbf{f}}{\to} (f_1, f_2,\ldots ,f_r)A \to 0$ of $A$-modules.
	In other words, $$\begin{pmatrix}f_1 & f_2 & \cdots & f_r\end{pmatrix} {\cdot} M = 0,$$ 
	so $\begin{pmatrix}f_1 & f_2 & \cdots & f_r\end{pmatrix} \in \Lambda$.
	Hence, $\image \varphi = (f_1,f_2,\ldots , f_r)A \subseteq \sum_{\mathbf{f} \in \Lambda} (f_1, f_2,\ldots ,f_r)A$ as required.
\end{proof}


\begin{Corollary}[cf. {\cite[Corollary 3.2]{HHS}}]\label{Comp_canontrace}
	Let $S$ be a Gorenstein local ring with the maximal ideal $\fkn$. Let $\fka$ be an ideal of $S$. Assume that $R=S/\fka$ is a Cohen-Macaulay local ring. We suppose there exists a minimal $S$-free resolution 
	\begin{align}\label{eq1}
	0 \to S^{\oplus r} \overset{{}^t\!M}{\longrightarrow} S^{\oplus q} \to \cdots \to S \to R \to 0
	\end{align}
	of $R$, where $0 < q,r \in \mathbb{Z}$, $M \in \Mat_{r \times q}(S)$, and ${}^t\!M$ denotes the transpose of $M$. We put
	$$
	\Lambda = \left\{\mathbf{f} = \begin{pmatrix}f_1 & f_2 & \cdots & f_r \end{pmatrix} \in \Mat_{1 \times r}(S) ~\middle|~\text{$\mathbf{f}{\cdot}M = 0$ in $\Mat_{1 \times q}(R)$}\right\}.
	$$
	Then $$\trace_R(\canon_R) = \sum_{(f_1~f_2~\cdots~f_r) \in \Lambda} (f_1, f_2,\ldots , f_r)R.$$
\end{Corollary}

\begin{proof}
Observe that $\canon_S \cong S$ because $S$ is Gorenstein. Taking the $S$-dual of the exact sequence \eqref{eq1}, we obtain the exact sequence
$$
S^{\oplus q} \to S^{\oplus r} \to \canon_R \to 0
$$
of $S$-modules. By applying the functor $R\otimes_S -$ to this exact sequence, we have the presentation
$$
R^{\oplus q} \to R^{\oplus r} \to \canon_R \to 0
$$
for the $R$-module $\canon_R$. 
Consequently, using Proposition \ref{Kumashiro}, we get that $$\trace_R(\canon_R) = \sum_{\mathbf{f} \in \Lambda} (f_1, f_2,\ldots ,f_r)R.$$
\end{proof}

\begin{Remark}\label{new2.25}
Suppose, in addition to the hypothesis of Corollary \ref{Comp_canontrace}, that $R$ is a domain. Then for each $(f_1~f_2~\cdots~f_r) \in \Lambda$, either $(f_1, f_2,\ldots , f_r)R \cong \canon_R$ or $(f_1, f_2,\ldots , f_r)R=0$ holds.
\end{Remark}

\begin{proof}
For the presentation $R^{\oplus q} \to R^{\oplus r} \to \canon_R \to 0$, we have the commutative diagram
$$
	\xymatrix{
	R^{\oplus q} \ar[r] \ar[rd]_{\text{zero map}} & R^{\oplus r} \ar[r] \ar[d] & \canon_R \ar[r] & 0\\
	&(f_1,f_2, \ldots ,f_r)R& &
	}
	$$
By the universal property of kernel, we have a canonical surjection $\varepsilon: \canon_R \to (f_1,f_2, \ldots ,f_r)R$. If $(f_1, f_2,\ldots , f_r)R$ is nonzero, the ranks of $\canon_R$ and $(f_1,f_2, \ldots ,f_r)R$ are one; hence, the rank of $\Ker \varepsilon$ is zero. This implies that $\Ker \varepsilon=0$ since $\canon_R$ is torsionfree. Thus, $\varepsilon$ is an isomorphism. 
\end{proof}

If $A$ is a graded ring and $X$ is a finitely generated graded $A$-module, $\trace_A(X)$ is a graded ideal of $A$. However, even for a homogeneous element $a \in \trace_A(X)$, it is not necessarily true that there exists $\mathbf{f} =\begin{pmatrix}f_1 & f_2 & \cdots & f_r\end{pmatrix}\in \Lambda$ such that $a \in \{f_1, f_2,\ldots , f_r\}$, where $\Lambda$ is as in Proposition \ref{Kumashiro}.
As for the graded case, we have the following.

\begin{Proposition}\label{graded}
	Let $A=\bigoplus_{\mathbf{a} \in \mathbb{Z}^n} A_{\mathbf{a}}$ be a $\mathbb{Z}^n$-graded ring such that $A_\mathbf{0} = k$ is a field and $X$ a non-zero $\mathbb{Z}^n$-graded $A$-module, where $\mathbf{0}$ denotes the zero vector. 
	Assume that there exists an exact sequence
	$G \overset{M}{\longrightarrow} F \overset{\varepsilon}{\to} X \to 0$
	of graded $A$-modules where $F$ and $G$ are finitely generated graded free $A$-modules of rank $q$ and $r$, and $M \in \Mat_{r \times q}(A)$ with homogeneous entries.
	Let $\mathbf{a} \in \mathbb{Z}^n$ such that $\dim_k A_\mathbf{a} = \ell<\infty$.
	If $A_{\mathbf{a}} \subseteq \trace_A(X)$, then there exists $U=(u_{ij}) \in \Mat_{\ell \times r}(A)$ such that
	\begin{enumerate}[{\rm (a)}]
		\item for all $1\le i \le \ell$, there exists $1 \le j_i \le r$ such that $u_{ij_i} \in A_{\mathbf{a}}$,
		\item $A_{\mathbf{a}} = \left< u_{ij_i} \mid 1 \le i \le \ell\right>_k$, and
		\item $U{\cdot}M = 0$ in $\Mat_{\ell \times q}(A)$.
	\end{enumerate}
	Here, $\left<u_1, u_2,\ldots , u_n\right>_k$ denotes the $k$-vector space spanned by the basis $\{u_1, u_2, \ldots , u_n\}$.
\end{Proposition}

For the proof of Proposition \ref{graded}, we prepare the following lemma.

\begin{Lemma}\label{grd_lem}
	Let $V$ be a $k$-vector space of dimension $\ell < \infty$. We take a $k$-basis $\{v_i \mid 1\le i \le \ell\}$ of $V$. Suppose that for each $1\le i \le \ell$, there exists an equality
	$$
	v_i = u_{i1} + u_{i2} + \cdots + u_{ip_i}
	$$
	such that $p_i \ge 1$ and $u_{ij} \in V$ for all $1 \le j \le p_i$.
	Then, for each $1 \le i \le \ell$, we can find $1 \le j_i\le p_i$ such that $V=\left<u_{1j_1}, u_{2j_2},\ldots , u_{\ell j_\ell}\right>_k$ holds true.
\end{Lemma}

\begin{proof}
	We prove by induction on $\ell = \dim_k V$. If $\ell = 1$, then it is obvious.
	Suppose that $\ell \ge 2$ and the assertion holds for $\ell-1$. 
	Put $\overline{V} = V/\left<v_\ell\right>_k = \left<\overline{v_1}, \ldots , \overline{v_{\ell -1}}\right>_k$. Then we have
	$$
	\overline{v_i} = \overline{u_{i1}} + \cdots + \overline{u_{ip_i}}
	$$
	for each $1 \le i \le \ell-1$ in $\overline{V}$. 
	By the induction hypothesis, there exist $j_1, j_2,\ldots , j_{\ell-1}$ such that
	$$\overline{V} = \left<\overline{u_{ij_i}} \mid 1 \le i \le \ell-1\right>_k.$$
	Since $\overline{v_\ell} \ne 0$ in $\overline{V}$ and $\dim_k \overline{V} = \ell-1$, we have $V=\left<u_{ij_i} \mid 1 \le i \le \ell-1\right>_k + \left<v_\ell\right>_k$. 
	Hence $\overline{v_\ell} \ne 0$ in $V' = V /\left<u_{ij_i} \mid 1 \le i \le \ell-1\right>_k$, hence there exists $j_{\ell}$ such that $\overline{u_{\ell j_\ell}} \ne 0$ in $V'$.
	Therefore $V=V'+\left<u_{\ell j_\ell}\right>_k = \left<u_{1j_1}, u_{2j_2},\ldots , u_{\ell j_\ell}\right>_k$ as desired.
\end{proof}

\begin{proof}[Proof of Proposition \ref{graded}]
Since $\dim_k  A_{\mathbf{a}} = \ell$, we may write $A_{\mathbf{a}} = \left<v_1, v_2,\ldots , v_\ell\right>_k$. 
Since $A_{\mathbf{a}} \subseteq \trace_A(X)$, we can find homogeneous elements $f_1, f_2,\ldots ,f_p \in \Hom_A(X,A)$ such that $v_1, v_2,\ldots , v_\ell \in \sum_{i=1}^p \image f_i$. We write
$$
\left\{
\begin{array}{c}
	v_1 = f_1(x_{11}) + f_2(x_{12}) + \cdots\quad + f_p(x_{1p})\\
	v_2 = f_1(x_{21}) + f_2(x_{22}) + \cdots\quad + f_p(x_{2p})\\
	\vdots\\
	v_\ell = f_1(x_{\ell 1}) + f_2(x_{\ell 2}) + \cdots\quad + f_p(x_{\ell p})\\
\end{array}
\right.
$$
for some homogeneous elements $x_{ij} \in X$. 
Then, for all $1 \le i \le \ell$ and $1 \le j \le p$, we may assume $f_j(x_{ij}) \in A_{\mathbf{a}}$ since $f_j$ and $x_{ij}$ are homogeneous.
Therefore, by Lemma \ref{grd_lem}, we can choose $1 \le j_1, j_2,\ldots , j_\ell \le p$ such that
$\{f_{j_1}(x_{1j_1}), f_{j_2}(x_{2j_2}), \ldots , f_{j_\ell}(x_{\ell j_\ell})\}
$ is a $k$-basis of $A_{\mathbf{a}}$. 
For $1 \le i \le \ell$, we put $g_i = f_{j_i}$ and $x_i = x_{i j_i}$. Namely, 
$$
A_{\mathbf{a}} = \left<g_1(x_1) , g_2(x_2),\ldots , g_\ell(x_\ell)\right>_k.
$$
Therefore, there exists an invertible matrix $C=(c_{ij}) \in \Mat_{\ell \times \ell}(k)$ such that
$$
\left\{
\begin{array}{c}
	g_1(x_1) = c_{11}v_1 + c_{12}v_2 + \cdots + c_{1\ell}v_\ell \\
	g_2(x_2) = c_{21}v_1 + c_{22}v_2 + \cdots + c_{2\ell}v_\ell \\
	\vdots\\
	g_\ell(x_\ell) = c_{\ell 1}v_1 + c_{\ell 2}v_2 + \cdots + c_{\ell \ell}v_\ell. \\
\end{array}
\right.
$$
On the other hand, let $\{\mathbf{e}_i\}_{1\le i \le r}$ be the standard basis of $F$ and put $y_i = \varepsilon(\mathbf{e}_i) \in X$. 
For all $1 \le i \le \ell$, we write $x_i = a_{i1}y_1 + a_{i2}y_2 + \cdots + a_{ir}y_r$ for some homogeneous elements $a_{ij} \in A$. We then have
$$
\left\{
\begin{array}{c}
	c_{11}v_1 + c_{12}v_2 + \cdots + c_{1\ell}v_\ell = a_{11}g_1(y_1) + a_{12}g_1(y_2) + \cdots + a_{1r}g_1(y_r)\\
	c_{21}v_1 + c_{22}v_2 + \cdots + c_{2\ell}v_\ell  = a_{21}g_2(y_1) + a_{22}g_2(y_2) + \cdots + a_{2r}g_2(y_r)\\
	\vdots\\
	c_{\ell 1}v_1 + c_{\ell 2}v_2 + \cdots + c_{\ell \ell}v_\ell = a_{\ell 1}g_\ell(y_1) + a_{\ell 2}g_\ell(y_2) + \cdots + a_{\ell r}g_\ell (y_r). \\
\end{array}
\right.
$$
Hence, there exists $(i_1,j_1), (i_2, j_2),\ldots ,(i_\ell, j_\ell)$ such that
$$
A_{\mathbf{a}} = \left< a_{i_\alpha j_\alpha} g_{i_\alpha}(y_{j_\alpha}) \mid 1 \le \alpha \le \ell \right>_k
$$
by Lemma \ref{grd_lem}. For $1 \le \alpha \le \ell$ and $1 \le \beta \le r$, we put
$$
u_{\alpha \beta} = a_{i_\alpha j_\beta}g_{i_\alpha}(y_\beta).
$$
Then $u_{\alpha j_\alpha} = a_{i_\alpha j_\alpha}g_{i_\alpha}(y_{j_\alpha})$ and $\mathbf{u} = \begin{pmatrix} u_{\alpha 1} & u_{\alpha 2} & \cdots & u_{\alpha r} \end{pmatrix}$ satisfy $\mathbf{u}{\cdot}M = 0$ in $\Mat_{1 \times q}(A)$, because $g_{i_\alpha} \circ \varepsilon = \widehat{a_{i_\alpha j_\alpha}} \circ \mathbf{u}$.
Thus, by putting $U=(u_{ij})$, the conditions (a), (b), and (c) are satisfied.
\end{proof}

\begin{Corollary}\label{graded dim1}
	Under the same notation as in Proposition \ref{graded}, let $\mathbf{a} \in \mathbb{Z}^n$ such that $\dim_k A_{\mathbf{a}} =1$ and $0 \ne u \in A_{\mathbf{a}}$. If $u \in \trace_A(X)$, then there exist $u_1, u_2,\ldots , u_r\in A$ such that $u \in \{u_1, u_2,\ldots , u_r\}$ and $\begin{pmatrix}u_1 & u_2 & \cdots & u_r\end{pmatrix}{\cdot}M = 0$ in $\Mat_{1 \times q}(A)$.
\end{Corollary}

\begin{Corollary}\label{graded dim2}
		Under the same notation as in Proposition \ref{graded}, let $\mathbf{a} \in \mathbb{Z}^n$ such that $\dim_k A_{\mathbf{a}} =2$. Put $A_{\mathbf{a}} = \left<x,y\right>_k$. If $A_{\mathbf{a}} \subseteq \trace_A(X)$, then there exist $a,b,c,d \in k$ and $u_1, u_2,\ldots , u_r, v_1, v_2,\ldots v_r \in A$ such that \begin{enumerate}[{\rm (a)}]
 	\item $ad-bc \ne 0$,
 	\item $ax+by \in \{u_1, u_2,\ldots , u_r\}$,
 	\item $cx+dy \in \{v_1, v_2,\ldots , v_r\}$, and 
 	\item $\begin{pmatrix}u_1 & u_2 & \cdots & u_r\\v_1 & v_2 & \cdots & v_r\end{pmatrix}{\cdot}M = 0$ in $\Mat_{2 \times q}(A)$.
 \end{enumerate}

\end{Corollary}



\section{Nearly Gorenstein property of numerical semigroup rings \\with determinantal defining ideals} \label{1dim}

Let $k$ be a field and $n \ge 3$ an integer. Take $0 < a_1, a_2,\ldots , a_n \in \mathbb{N}$ such that the greatest common divisor of $a_1, a_2,\ldots , a_n$ is $1$, where $\mathbb{N}$ denotes the set of all non-negative integers.
We consider a numerical semigroup 
$$
H = \left< a_1, a_2,\ldots ,a_n\right> = \left\{\sum_{i=1}^n \lambda_i a_i \ \middle| \  \lambda_i \in \mathbb{N}\right\}.
$$
Throughout this article, we always assume that $a_1, a_2,\ldots ,a_n$ is the minimal system of generators of $H$. 
Then, the $\mathbb{Z}$-graded $k$-subalgebra $A=k[H] = k[t^h \mid h \in H]$ of the polynomial ring $k[t]$ with the standard grading, is called the {\it numerical semigroup ring} of $H$ over $k$. 
Let $B=k[X_1,X_2,\ldots ,X_n]$ be the polynomial ring over $k$ with $n$ indeterminates. We regard $B$ as a $\mathbb{Z}$-graded ring with $B_0 = k$ and $\deg X_i = a_i$.
The graded $k$-algebra homomorphism $\varphi : B \to A$ defined by $\varphi(X_i) = t^{a_i}$ for every $1 \le i \le n$ induces a $B$-module structure of any $A$-modules.
By applying Corollary \ref{graded dim1} to $A$, and further combining Corollary \ref{Comp_canontrace}, we readily get the following, since $\dim_k A_p \le 1$ for all $p \in \mathbb{Z}$.

\begin{Proposition}\label{3.1}
	Let $0 \to F_p \overset{{}^tM}{\to} F_{p-1} \to \cdots \to F_0 \overset{\varepsilon}{\to} A \to 0$ be the graded minimal $B$-free resolution of $A$ with $\rank_B F_p = r$ and put
	$$
	\Lambda = \{\mathbf{f} = \begin{pmatrix} f_1 & f_2 & \cdots & f_r \end{pmatrix} \in \Mat_{1 \times r} (B) \mid \mathbf{f}{\cdot}M = 0 \text{ in $\Mat_{1 \times r}(A)$}\}.
	$$
	Let $u\in H$. Then, $t^u \in \trace_A(\canon_A)$ if and only if there exists $\mathbf{f} = \begin{pmatrix}f_1 & f_2 & \cdots & f_r\end{pmatrix} \in \Lambda$ such that $t^u \in \{\varepsilon(f_1), \varepsilon(f_2),\ldots ,\varepsilon(f_r)\}$.
\end{Proposition}

Now, we consider
\[
R=\widehat{A_\fkM}=k[[t^h \mid h\in H]] \subseteq k[[t]] = V,
\] 
where $\fkM$ denotes the graded maximal ideal $(t^{a_1}, t^{a_2}, \dots, t^{a_n})$ of $A$, $\widehat{*}$ denotes the $\fkm = \fkM A_\fkM$-adic completion, and $k[[t]]$ denotes the formal power series ring over $k$. We put $S=k[[X_1, X_2,\ldots , X_n]]$ the formal power series ring with $n$ indeterminates.
$R$ is also called the {\it numerical semigroup ring} of $H$ over $k$ and denoted by $k[[H]]$.
Let $I_H$ be the kernel of the canonical surjective $k$-algebra homomorphism $S \to R$ obtained by $X_i\mapsto t^{a_i}$. Hence $R\cong S/I_H$. We denote by $x_i=t^{a_i}$ the image of $X_i$ in $R$.
Note that, for $u \in H$, $t^u \in \trace_R(\canon_R)$ if and only if $t^u \in \trace_A(\canon_A)$, where $\canon_A$ denotes the graded canonical module of $A$.

In what follows, we always assume that $R$ has a certain determinantal defining ideal; that is, $I_H$ has the following form:
$$I_H=\detid_2\begin{pmatrix}
		X_2^{m_2} & X_3^{m_3} & \cdots & X_n^{m_n} & X_1^{m_1}\\
		X_1^{\ell_1} & X_2^{\ell_2} & \cdots & X_{n-1}^{\ell_{n-1}} & X_n^{\ell_n}
		\end{pmatrix}$$
for some positive integers $m_1, m_2,\ldots, m_n, \ell_1,\ell_2,\ldots , \ell_n$. Here, $\detid_2(\mathcal{M})$ denotes the ideal of $S$ generated by $2$-minors of a matrix $\mathcal{M}$ with entries in $S$.
For an integer $r \in \mathbb{Z}$, $\nint{r}$ denotes the integer such that $1 \le \nint{r} \le n$ and $\nint{r} \equiv r \mod n$.
Then we note that $\{X_{\nint{i+1}}^{m_{\nint{i+1}}} X_j^{\ell_j} - X_{\nint{j+1}}^{m_{\nint{j+1}}} X_i^{\ell_i} \mid 1\le i< j \le n\}$ is a minimal system of generators of $I_H$.

Thanks to the work of Eagon-Northcott \cite{EN}, we can get a minimal $S$-free resolution of $R$ of the form $0 \to S^{\oplus (n-1)} \overset{{}^t\! M}{\longrightarrow}  S^{\oplus n(n-2)} \to \cdots \to S \to R \to 0$, where

	\[
	M =\begin{pmatrix}
	X_2^{m_2}  & \cdots & X_1^{m_1} & \\
	-X_1^{\ell_1} & \cdots & -X_n^{\ell_n} & X_2^{m_2}  & \cdots & X_1^{m_1} & \\
	 &  &  & -X_1^{\ell_1} & \cdots & -X_n^{\ell_n} & \\
	 & & & & & & \ddots \\
	 & & & & & & &  X_2^{m_2}  & \cdots & X_1^{m_1}\\
	 & & & & & & &  -X_1^{\ell_1}  & \cdots & -X_n^{\ell_n}\\
	\end{pmatrix},
	\]
	by swapping the signatures in some columns and all even-numbered rows. Let $N \in \Mat_{(n-1) \times n(n-2)}(R)$ such that each entry of $N$ is the image of the corresponding entry of $M$. We put $x_i$ the image of $X_i$ in $R$ for $1 \le i \le n$.
	We denote the $i$-th column of $N$ by $N_i$ and define
	$$
	\Lambda = \left\{\mathbf{f} = \begin{pmatrix}f_1 & f_2 & \cdots & f_{n-1} \end{pmatrix} \in \Mat_{1 \times (n-1)}(R) ~\middle|~\text{$\mathbf{f}{\cdot}N = 0$ in $\Mat_{1 \times (n(n-2))}(R)$}\right\}.
	$$
	Namely, by denoting $\left[N_{(j-1)n+i}\right]_\ell$ the $\ell$-th entry of the column vector $N_{(j-1)n+i}$, we have 
	$$
	\left[N_{(j-1)n+i}\right]_\ell = \begin{cases}
 		x_{\nint{i+1}}^{m_{\nint{i+1}}} & (\ell = j)\\
 		-x_{i}^{\ell_i} & (\ell = j+1)\\
 		0 & (\ell \ne j, j+1).
	\end{cases}
	$$
	Hence, 
\begin{align}\label{eq2}
	x_{\nint{i+1}}^{m_{\nint{i+1}}} f_j - x_i^{\ell_i} f_{j+1} = 0 \text{ in $R$}
\end{align}
for all $\mathbf{f} = \begin{pmatrix} f_1 & f_2 & \cdots & f_{n-1}\end{pmatrix} \in \Lambda$, $1 \le i \le n$, and $1 \le j \le n-2$.
From now on, for $x\in R$ and $\mathbf{f} = \begin{pmatrix} f_1 & f_2 & \cdots & f_{n-1}\end{pmatrix}\in \Lambda$, $x\in \mathbf{f}$ denotes $x=f_r$ for some $1\le r \le n-1$.
By the form of \eqref{eq2}, we get the following key lemma. 

\begin{Lemma}\label{key}
	Let $\mathbf{f} \in \Lambda$ and $1 \le i \le n$. We have the following.
	\begin{enumerate}[\rm(1)] 
		\item Suppose $m_i \ge 2$. If $x_i \in \mathbf{f}$, then $f_1 = x_i$.
		\item Suppose $\ell_i \ge 2$. If $x_i \in \mathbf{f}$, then $f_{n-1} = x_i$.
	\end{enumerate}
\end{Lemma}

\begin{proof}
	Throughout this proof, we may assume $i=1$.
	
	(1) Suppose $f_p = x_1$ for some $2 \le p \le n-1$.
	Then we have 
	$$
		x_1^{m_1} f_{p-1} - x_n^{\ell_n} x_1 = \mathbf{f}{\cdot}N_{(p-1)n} =0.
	$$
	Hence, since $R$ is an integral domain, we can cancel $x_1$ from both sides to get 
	$$x_1^{m_1-1}f_{p-1} = x_n^{\ell_n}.$$
	Then our assumption $m_1 \ge 2$ implies $X_n^{\ell_n} \in I_H+(X_1)$. 
	On the other hand, $X_n^{m_n + \ell_n} - X_{n-1}^{\ell_{n-1}} X_1^{m_1}$ belongs to a minimal system of generators of $I_H$. Therefore, $X_n^{m_n + \ell_n} \in I_H+(X_1)$ must also belong to a minimal system of generators of $I_H+(X_1)$.
	This leads to a contradiction that $m_n =0$.
	Therefore, we must have $f_1 = x_1$.

	(2) We can prove the assertion in the same way as (1) by considering $\mathbf{f}{\cdot}N_{(p-2)n + 1}$. 
\end{proof}

Consequently, we have the following.

\begin{Corollary}\label{key_cor}
	If $x_i \in \trace_R(\canon_R)$, then $m_i=1$ or $\ell_i = 1$.
\end{Corollary}

\begin{proof}
	Suppose $x_i \in \trace_R(\canon_R)$. 
	Then, thanks to Proposition \ref{3.1}, there exists  $\mathbf{f} \in \Lambda$ such that $x_i \in \mathbf{f}$.
	If $m_i, \ell_i \ge 2$, then $f_1 = x_i = f_{n-1}$ by Lemma \ref{key} which is a contradiction by Remark \ref{new2.25}. Therefore, we yield $m_i = 1$ or $\ell_i = 1$.
\end{proof}

The following is the main result of this section.

\begin{Theorem}\label{main}
	$R=k[[H]]$ is a nearly Gorenstein ring if and only if, after suitable permutation of the minimal system $a_1, a_2,\ldots , a_n$ of generators $H$, one of the following is satisfied.
\begin{enumerate}[\rm(1)] 
		\item $m_1 = m_2 = \cdots = m_n = 1$.
		\item $m_2 = m_3 = \cdots = m_n = \ell_1 = \ell_2 = \cdots = \ell_{n-2}=1$.
	\end{enumerate}
\end{Theorem}

\begin{proof}
	First, we prove the ``if'' part. In the case (1), $R$ is an almost Gorenstein ring (e.g., see \cite[Theorem 7.8]{GTT}). Moreover, thanks to \cite[Proposition 6.1]{HHS}, every $1$-dimensional almost Gorenstein ring is nearly Gorenstein. Therefore we have done. Suppose the case (2). 
	Then, by direct computation, 
	$$
	\begin{pmatrix}
		x_1 & x_2 & \cdots & x_{n-1}
	\end{pmatrix}{\cdot}N = 0 \text{ and }
	$$
	$$
	\begin{pmatrix}
		x_2x_{n-1}^{\ell_{n-1}-1} & x_3 x_{n-1}^{\ell_{n-1}-1} & \cdots & x_{n-1}^{\ell_{n-1}} & x_n
	\end{pmatrix}{\cdot}N = 0
	$$
	in $\Mat_{1 \times (n(n-2))}(R)$. 
	This means 
	$$	\begin{pmatrix}
		x_1 & x_2 & \cdots & x_{n-1}
	\end{pmatrix}, \begin{pmatrix}
		x_2x_{n-1}^{\ell_{n-1}-1} & x_3 x_{n-1}^{\ell_{n-1}-1} & \cdots & x_{n-1}^{\ell_{n-1}} & x_n
	\end{pmatrix} \in \Lambda.$$ 
	Therefore, by Corollary \ref{Comp_canontrace}, $(x_1, \ldots , x_n) \subseteq \trace_R(\canon_R)$. Thus, $R$ is a nearly Gorenstein ring.
	
	To prove the  ``only if'' part, we assume there exists a numerical semigroup $H=\left<a_1, a_2,\ldots , a_n\right>$ with the defining ideal $I_H=\detid_2\begin{pmatrix}
		X_2^{m_2} & X_3^{m_3} & \cdots & X_n^{m_n} & X_1^{m_1}\\
		X_1^{\ell_1} & X_2^{\ell_2} & \cdots & X_{n-1}^{\ell_{n-1}} & X_n^{\ell_n}
	\end{pmatrix}$ such that $R=k[[H]]$ is nearly Gorenstein, but there is no rearrangement of $a_1, a_2,\ldots ,a_n$ satisfying (1) or (2). 
	Since $k[[H]]$ is an integral domain, we may assume that $m_1 \ge 2$ after some permutation of $a_1, a_2,\ldots , a_n$. 
	Since $R$ is nearly Gorenstein, $x_1 \in \trace_R(\canon_R)$, whence $\ell_1 = 1$ by Corollary \ref{key_cor}.
	As $R$ does not satisfy (1) for any permutation of $a_1, a_2, \dots, a_n$, there exists $2 \le p \le n$ such that $\ell_p \ge 2$. 
	We take such $p$ as small as possible. 
	Then $\ell_1 = \ell_2 = \cdots = \ell_{p-1} =1$ and $\ell_p \ge 2$. 
	On the other hand, $x_p \in \trace_R(\canon_R)$ since $R$ is nearly Gorenstein. 
	Hence, by Proposition \ref{3.1} and Lemma \ref{key} (2), there exists $f_1, f_2, \ldots , f_{n-2} \in R$ such that
	$\begin{pmatrix}f_1 & \cdots & f_{n-2} & x_p\end{pmatrix}{\cdot}N = 0$ in $\Mat_{1\times (n(n-2))} (R)$. We put $f_{n-1} = x_p$.
	\begin{claim}
	 We have the following.
	\begin{enumerate}[{\rm (i)}]
		\item If $2 \le p \le n-1$, then $m_i=1$ and $f_{n-1+i-p} = x_i$ for all $2 \le i \le p$.
		\item If $p = n$, then $m_i = 1$ and $f_{i-1}=x_i$ for all $3 \le i \le n$.
	\end{enumerate}
	\end{claim}
	\begin{proof}[Proof of Claim 1]
		(i) We prove this by descending induction on $i$. By Corollary \ref{key_cor}, we have $m_p = 1$. Suppose $2 \le i < p$, $m_{i+1} = 1$, and $f_{n+i-p} = x_{i+1}$.
		We then get
		$$
		f_{n-1+i-p}x_{i+1} - x_{i+1}x_i = 0
		$$
		because $\ell_i = 1$ and $m_{i+1} = 1$. 
		Since $R$ is an integral domain, $f_{n-1+i-p} - x_i =0$ which implies $f_{n-1+i-p} = x_i$.
		Moreover, by considering $\begin{pmatrix}f_1 & \cdots & f_{n-1}\end{pmatrix}{\cdot} N_{n(n-4) +i}$, we have
		$f_{n-2+i-p}x_{i+1} - x_i^2 =0$. Therefore $m_i = 1$, because $X_i^{m_i+1} - X_{i-1}X_{i+1}$ is a part of a minimal system of generators of $I_H$.
		
		(ii) follows by the same argument as (i).
	\end{proof}
 
 	Hence we have $p \le n-1$, since if $p=n$, then (2) is satisfied by Claim 1 (ii).

 	Notice that, by Proposition \ref{3.1} and Lemma \ref{key} (2), we can find $g_2, g_3,\ldots , g_{n-1} \in R$ such that 
 	$
 	\begin{pmatrix} x_1 & g_2 & \cdots & g_{n-1}\end{pmatrix} \in \Lambda
 	$.
 	Furthermore, we get 
 	$$
 	g_2 = x_2, g_3 = x_3, \ldots , g_p = x_p
 	$$
 	by considering $\begin{pmatrix} x_1 & g_2 & \cdots & g_{n-1}\end{pmatrix}{\cdot}N_{(i-1)n + i}$ for $1 \le i \le p-1$, recursively.
 	Suppose $p \le n-2$. Then 
 	$$x_p^2 - g_{p+1} x_{p-1} = \begin{pmatrix} x_1 & g_2 & \cdots & g_{n-1}\end{pmatrix}{\cdot}N_{(p-1)n + p} = 0$$
 	whence $\ell_p=1$, which is a contradiction. 
 	Therefore $p = n-1$.
 	In this case, we have
 	$$
 	I_H=\detid_2\begin{pmatrix}
	X_2 & \cdots & X_{n-1} & X_n^{m_n} & X_1^{m_1}\\
	X_1 & \cdots & X_{n-2} & X_{n-1}^{\ell_{n-1}} & X_n^{\ell_n}
	\end{pmatrix}
 	$$
 	and $m_n = 1$ or $\ell_n = 1$ by Corollary \ref{key_cor}. In each case, we can rearrange $a_1,a_2,\ldots , a_n$ satisfying condition (2). But this is absurd.
 	This completes the proof of Theorem \ref{main}.
\end{proof}

\begin{Remark}
	The case (1) in Theorem \ref{main}, it is know by \cite[Theorem 7.8]{GTT} that $k[[H]]$ is an almost Gorenstein local ring in the sense of Goto-Takahashi-Taniguchi \cite{GTT}.
	This is equivalent to the condition that the numerical semigroup $H$ is almost symmetric.
	Under our setting, case (2) in Theorem \ref{main} corresponds to nearly Gorenstein rings that are not almost Gorenstien.
	Theorem \ref{main} seems to suggest that such nearly Gorenstein rings are rare as there are only 3 parameters $m_1, \ell_{n-1},\ell_n$ that can be freely chosen.
	We recall that every almost Gorenstein local ring of dimension $1$ is nearly Gorenstein as established by \cite[Proposition 6.1]{HHS}. 
\end{Remark}

\begin{Corollary}
	Suppose that $R=k[[H]]$ is a nearly Gorenstein local ring that is not almost Gorenstein, then $a_1, a_2,\ldots , a_{n-1}$ forms an arithmetic progression, after suitable permutation of the minimal system $a_1, a_2,\ldots ,a_n$ of generators of $H$.
\end{Corollary}

\begin{proof}
	Thanks to Theorem \ref{main}, we know that the defining ideal of $R$ is given by
	$$
	I_H=\detid_2 \begin{pmatrix}
 		X_2 & \cdots & X_i & \cdots & X_n & X_1^{m_1}\\
 		X_1 & \cdots & X_{i-1} & \cdots & X_{n-1}^{\ell_{n-1}} & X_n^{\ell_n}
 	\end{pmatrix},
	$$
	after suitable permutation of $a_1, a_2,\ldots ,a_n$. 
	Notice that the defining ideal of $A=k[H]$ also has the same form, thus being a graded ideal. 
	This observation implies that $a_i - a_{i-1}$ is constant for all $2 \le i \le n-1$ since the difference between the degrees of entries in the first row and the second row is constant for each column (see, e.g., \cite[Lemma 1]{KM}).
	Therefore, $a_1, a_2,\ldots ,a_{n-1}$ forms an arithmetic progression.
\end{proof}

\begin{Remark}
	For given positive integers, $m_1, m_2, \ldots , m_n, \ell_1, \ell_2, \ldots , \ell_n$, we can regard the ring
	$$
	A=k[X_1, X_2,\ldots , X_n] /\detid_2 \begin{pmatrix} X_2^{m_2} & X_3^{m_3} & \cdots & X_n^{m_n} & X_1^{m_1}\\
 	X_1^{\ell_1} & X_2^{\ell_2} & \cdots & X_{n-1}^{\ell_{n-1}} &  X_n^{\ell_n}
 	\end{pmatrix}
	$$
	as a graded ring by putting
	$$
	\deg X_i = \sum_{j=1}^n (m_{\nint{i+1}} m_{\nint{i+2}} \cdots m_{\nint{i+j-1}}) (\ell_{\nint{i+j}} \ell_{\nint{i+j+1}} \cdots \ell_{\nint{i+n-1}}).
	$$
	Here, $\ell_{\alpha}\ell_{\alpha+1} \cdots \ell_{\beta} = m_{\alpha}m_{\alpha+1} \cdots m_{\beta} =1$, if $\alpha > \beta$.
	In general, $\dim_k A_n$ can be larger than $1$ for $1 \le n \in \mathbb{N}$. 
	Even for such cases, similar results as Theorem \ref{main} should be achievable. 
	However, addressing these cases tends to complicate the description, making the discussion in the following section more intricate.
	This is why we have chosen to focus on numerical semigroup rings throughout this section.
\end{Remark}

\begin{Example}
	Let $m \ge 3$ be a positive integer with $m \not\equiv 2 \mod 7$. Put $H=\left<7, m+5, 2m+3, 3m+1\right>$. Then $k[[H]]$ is a nearly Gorenstein local ring but not almost Gorenstein.
	In fact, we have 
	$$
	k[[H]] \cong k[[X_1,X_2,X_3, X_4]] / \detid_2
	\begin{pmatrix}
		X_2 & X_3 & X_4 & X_1^m\\
		X_1 & X_2 & X_3 & X_4^2
	\end{pmatrix}.
	$$
\end{Example}

\section{Higher-dimensional nearly Gorenstein local rings}

We use the same notation as in Section \ref{1dim} and assume the defining ideal of $R=k[[H]]$ has the form
$$
\detid_2 \begin{pmatrix} X_2^{m_2} & X_3^{m_3} & \cdots & X_n^{m_n} & X_1^{m_1}\\
 	X_1^{\ell_1} & X_2^{\ell_2} & \cdots & X_{n-1}^{\ell_{n-1}} & X_n^{\ell_n}
 	\end{pmatrix}
$$
for some positive integers $m_1, m_2,\ldots , m_n, \ell_1, \ell_2, \ldots , \ell_n$.
For non-negative integers $p,q$ and
$$
1 \le i_1 < i_2 < \cdots < i_p \le n, ~~ 1 \le j_1 < j_2 < \cdots < j_q \le n,
$$
we put $I=\{i_1,i_2,\ldots , i_p\}$, $J=\{j_1, j_2,\ldots , j_q\}$, and 
$$
S^{I}_{J} = k[[X_1, X_2,\ldots ,X_n, Y_{i_1},Y_{i_2},\ldots , Y_{i_p}, Z_{j_1}, Z_{j_2}, \ldots , Z_{j_q}]]
$$
the formal power series ring with $n+p+q$ indeterminates over the field $k$.
Moreover, we consider 
$$D^I_J =\begin{pmatrix} V_2 & \cdots &  V_n & V_1\\
 	U_1 & \cdots & U_{n-1} & U_{n}
 	\end{pmatrix}
 	\in \Mat_{2 \times n}(S^I_J), \text{ where}$$ 
$$
V_{r} = \begin{cases}
	X_{r}^{m_{r}} + Y_{r} & (r \in I)\\
	X_{r}^{m_{r}} & (\text{otherwise})\\
\end{cases}
\quad \quad \text{and} \quad \quad 
U_{r} = \begin{cases}
	X_{r}^{\ell_{r}} + Z_{r} & (r \in J)\\
	X_{r}^{\ell_{r}} & (\text{otherwise})
\end{cases}
$$
for $1 \le r \le n$.
With the above notation, we consider the ring
$$
R^I_J = S^I_J / \detid_2\left(D^I_J\right).
$$


\begin{Remark}\label{new4.0}
\begin{enumerate}[\rm(1)] 
\item $R^I_J$ is a Cohen-Macaulay local domain of dimension $p+q+1$. Moreover, the sequence $x_r, y_{i_1}, y_{i_2}, \ldots, y_{i_p}, z_{j_1}, z_{j_2},\ldots , z_{j_q}$ forms an $R^I_J$-regular sequence for all $1 \le r \le n$.
\item Let $1 \le i \le n$. If there exists a homogeneous element $\varphi\in R^I_J$ such that $u_i, v_{\nint{i+1}} \in (\varphi)$, then $(\varphi)=R^I_J$.
\end{enumerate}
\end{Remark}

\begin{proof}
(1) By construction of $\detid_2\left(D^I_J\right)$, $\dim R^I_J \le n+p+q-(n-1)=p+q+1$. On the other hand, it is straightforward to check $\{x_r, y_{i_1}, y_{i_2}, \ldots, y_{i_p}, z_{j_1}, z_{j_2},\ldots , z_{j_q}\}$ is a system of parameter of $R^I_J$. It follows that the height of $\detid_2\left(D^I_J\right)$ is $n-1$. Hence, by \cite{EN}, $R^I_J$ is a Cohen-Macaulay local ring of dimension $p+q+1$. Since $R^I_J/(y_{i_1}, y_{i_2}, \ldots, y_{i_p}, z_{j_1}, z_{j_2},\ldots , z_{j_q})$ is a numerical semigroup ring, this residue ring is a domain. This implies that $R^I_J$ is also a domain by the following claim. 
\begin{claim}{\rm (\cite[page 62, Lemma]{Ge})}\label{claim2}
Let $(R, \fkm)$ be a Noetherian local ring. If $x\in \fkm$ is a non-zerodivisor of $R$ and $R/(x)$ is a domain, then $R$ is a domain.
\end{claim}

\begin{proof}[Proof of Claim \ref{claim2}]
This should be a well-known fact, but we record a proof here since it may be difficult to get the book \cite{Ge}. 
By the assumption, $(x)\in \Spec R\setminus \Ass R$. Thus there exists $\fkp\in \Spec R$ such that $\fkp\subsetneq (x)$. Write $\fkp=(a_1x, a_2x, \ldots, a_sx)$ for $a_1, \ldots,a_s\in R$. Since $\fkp$ is prime and $x\notin \fkp$, $a_i\in\fkp$ for each $1\le i \le s$. It follows that $\fkp=x\fkp$; hence, $\fkp=0$ by Nakayama's lemma. This shows that $R$ is a domain. 
\end{proof}

(2) We write $u_i = \varphi u_i'$ and $v_{[i+1]} = \varphi v_{[i+1]}'$ for $u_i', v_{[i+1]}' \in R^I_J$. 
Note that $U_iV_i - U_{\nint{i-1}}V_{\nint{i+1}}$ is a part of minimal system of generators of $\detid_2\left(D^I_J\right)$. 
Hence, $u_iv_i - u_{\nint{i-1}}v_{\nint{i+1}} = \varphi(u'_iv_i - u_{\nint{i-1}}v'_{\nint{i+1}}) = 0$ implies that either $\varphi=0$ or $\varphi$ is a unit since $R^I_J$ is a domain by (1). Since $0\ne u_i\in (\varphi)$, $\varphi$ is a unit of $R^I_J$.
\end{proof}

Moreover, notice that
$$
[S^I_J]^{\oplus (n-2)n} \overset{M^I_J}{\longrightarrow} [S^I_J]^{\oplus (n-1)} \longrightarrow \canon_{R^I_J} \to 0
$$
is a minimal graded presentation of the canonical module $\canon_{R^I_J}$, where
$$
	M^I_J =\begin{pmatrix}
	V_2  & \cdots & V_1 & \\
	-U_1 & \cdots & -U_n & V_2  & \cdots & V_1 & \\
	 &  &  & -U_1 & \cdots & -U_n & \\
	 & & & & & & \ddots \\
	 & & & & & & &  V_2  & \cdots & V_1\\
	 & & & & & & &  -U_1  & \cdots & -U_n\\
	\end{pmatrix}.
$$
We put $N^I_J \in \Mat_{(n-1)\times n(n-2)}(R^I_J)$ whose $(i,j)$-entries are the images of those of $M^I_J$. We define $v_i$ and $u_i$ as the image of $V_i$ and $U_i$ in $R^I_J$ respectively, for $1 \le i \le n$.

In this section, we explore the problem of when $R^I_J$ is nearly Gorenstein. It may be worth noting that such a construction of $R^I_J$ is helpful to consider higher-dimensional examples based on numerical semigroup rings; e.g., see \cite{EGGHIKMT} for quasi-normality, and \cite{GIKT} for almost Gorenstein property. 
 
First, we note the following lemma, which follows from Remark \ref{new4.0} (1) and \cite[Proposition 2.3 (b)]{HHS}.

\begin{Lemma}\label{new4.5}
Suppose that $R^I_J$ is nearly Gorenstein. 
\begin{enumerate}[\rm(1)] 
\item Suppose that $p>0$. Let $i\in I$. Then, $R^{I\setminus \{i\}}_J$ is also nearly Gorenstein.
\item Suppose that $q>0$. Let $j\in J$. Then, $R^I_{J\setminus \{j\}}$ is also nearly Gorenstein.
\end{enumerate}
\end{Lemma}

Lemma \ref{new4.5} leads us that if $R^I_J$ is nearly Gorenstein, then 
$R=k[[H]]$ is also nearly Gorenstein, whence, by Theorem \ref{main}, we can re-arrange the order of $X_1, X_2,\ldots ,X_n$ so that either
\begin{enumerate}[(a)]
	\item $m_1 = m_2 = \cdots = m_n = 1$ or
	\item $m_2 = m_3 = \cdots m_n = \ell_1 = \ell_2= \cdots =\ell_{n-2} = 1$.
\end{enumerate}

If $n=3$, then $\trace_{R^I_J} (\canon_{R^I_J})=I_1 (N^I_J)$ by \cite[Corollary 3.5]{HHS}. Thanks to this result, we can easily check the following two propositions.  

\begin{Proposition}\label{n3AGL}
	Suppose that $n = 3$ and $m_1=m_2=m_3=1$. Then $R^I_J$ is nearly Gorenstein if and only if $I\cap J=\emptyset$ and $\ell_i = 1$ for all $i \in I$.
\end{Proposition}

\begin{Proposition}\label{n3nonAGL}
	Suppose that $n = 3$, $m_2=m_3=\ell_1=1$ and $m_1 \ge 2$. Then $R^I_J$ is nearly Gorenstein if and only if $I\cap J=\emptyset$, $1 \notin J$, and $\ell_i = 1$ for all $i \in I$.
\end{Proposition}

The following two theorems answer the case of $n \ge 4$. For simplicity, we denote $R^I_J$ by $R^I$ (resp. $R_J$) if $J=\emptyset$ (resp. $I=\emptyset$). 

\begin{Theorem}\label{newAGcase}
	Suppose $n \ge 4$ and $m_1 = m_2 = \cdots = m_n = 1$. Then we have the following.
	\begin{enumerate}[{\rm (1)}]
		\item {\rm ($|I|+|J|=1$ case):} Let $1 \le i \le n$. 
		\begin{enumerate}[\rm(a)] 
		\item $R^{\{i\}}$ is nearly Gorenstein if and only if $\ell_{i} = \ell_{\nint{i+1}} = \cdots = \ell_{\nint{i+n-3}} = 1$.
		\item $R_{\{i\}}$ is nearly Gorenstein if and only if $\ell_{\nint{i+1}} = \ell_{\nint{i+2}} = \cdots = \ell_{\nint{i+n-3}} = 1$.
		\end{enumerate}
		
		\item {\rm ($|I|+|J|=2$ case):} 
		\begin{enumerate}[\rm(a)] 
		\item Let $1 \le i < j \le n$. $R_{\{i,j\}}$ is nearly Gorenstein if and only if $n=4$, $(i,j) \in \{(1,3), (2,4)\}$, and $\ell_{i+1} = \ell_{\nint{j+1}} = 1$.
		\item Let $1 \le i \le n$ and $1 \le j \le n$. $R^{\{i\}}_{\{j\}}$ is nearly Gorenstein if and only if $n=4$, $(i,j) \in \{(1,3), (2,4), (3,1),(4,2)\}$, and $\ell_{i} = \ell_{\nint{i+1}} = \ell_{\nint{j+1}} = 1$.
		\item For any $1 \le i < j \le n$, $R^{\{i,j\}}$ is not nearly Gorenstein.
		\end{enumerate}
		
		\item {\rm ($|I|+|J|\ge 3$ case):} In this case, $R^I_J$ can never become nearly Gorenstein.
	\end{enumerate}
\end{Theorem}

\begin{Theorem}\label{newnonAGcase}
	Suppose $n \ge 4$. Assume that $m_2 = m_3 = \cdots = m_n = \ell_1 = \ell_2=\cdots = \ell_{n-2}=1$, $m_1 \ge 2$, and ``$\ell_{n-1} \ge 2$ or $\ell_n \ge 2$''.
	Then we have the following.
	\begin{enumerate}[{\rm (1)}]
		\item {\rm ($|I|+|J|=1$ case):} Let $1 \le i \le n$. 
		\begin{enumerate}[\rm(a)] 
		\item $R^{\{i\}}$ is nearly Gorenstein if and only if either $i=1$ or ``$i=n$ and $\ell_n = 1$''.
		\item $R_{\{i\}}$ is nearly Gorenstein if and only if either $i=n$ or ``$i=n-1$ and $\ell_n = 1$''.
		\end{enumerate}
		
		\item {\rm ($|I|+|J|=2$ case):} 
		\begin{enumerate}[\rm(a)] 
		\item For any $1\le i < j \le n$, $R^{\{i,j\}}$ and $R_{\{i,j\}}$ is not nearly Gorenstein.
		\item Let $1 \le i \le n$ and $1 \le j \le n$. $R^{\{i\}}_{\{j\}}$ is nearly Gorenstein if and only if $n=4$, $i=1$, $j=3$, and $\ell_4 = 1$.
		\end{enumerate}
		
		\item {\rm ($|I|+|J|\ge 3$ case):} In this case, $R^I_J$ can never become nearly Gorenstein.
	\end{enumerate}
\end{Theorem}

Thanks to Lemma \ref{new4.5}, to prove Theorems \ref{newAGcase} and \ref{newnonAGcase}, it is enough to prove the following.

\begin{Theorem}\label{AGcase}
	Suppose $n \ge 4$ and $m_1 = m_2 = \cdots = m_n = 1$. Then we have the following.
	\begin{enumerate}[{\rm (1)}]
		\item Let $1 \le i \le n$. $R^{\{i\}}$ is nearly Gorenstein if and only if $\ell_{i} = \ell_{\nint{i+1}} = \cdots = \ell_{\nint{i+n-3}} = 1$.
		\item Let $1 \le i \le n$. $R_{\{i\}}$ is nearly Gorenstein if and only if $\ell_{\nint{i+1}} = \ell_{\nint{i+2}} = \cdots = \ell_{\nint{i+n-3}} = 1$.
		\item Let $1 \le i < j \le n$. $R_{\{i,j\}}$ is nearly Gorenstein if and only if $n=4$, $(i,j) \in \{(1,3), (2,4)\}$, and $\ell_{i+1} = \ell_{\nint{j+1}} = 1$.
		\item Let $1 \le i \le n$ and $1 \le j \le n$. $R^{\{i\}}_{\{j\}}$ is nearly Gorenstein if and only if $n=4$, $(i,j) \in \{(1,3), (2,4), (3,1),(4,2)\}$, and $\ell_{i} = \ell_{\nint{i+1}} = \ell_{\nint{j+1}} = 1$.
		\item For any $1 \le i < j \le n$, $R^{\{i,j\}}$ is not nearly Gorenstein.
		\item For any $1\le i  \le n$ and $1 \le j<k \le n$, $R_{\{j,k\}}^{\{i\}}$ is not nearly Gorenstein.
		\item For any $1 \le i < j < k \le n$, $R_{\{i,j,k\}}$ is not nearly Gorenstein.
	\end{enumerate}
\end{Theorem}

\begin{Theorem}\label{nonAGcase}
	Suppose $n \ge 4$. Assume that $m_2 = m_3 = \cdots = m_n = \ell_1 = \ell_2=\cdots = \ell_{n-2}=1$, $m_1 \ge 2$, and ``$\ell_{n-1} \ge 2$ or $\ell_n \ge 2$''.
	Then we have the following.
	\begin{enumerate}[{\rm (1)}]
		\item Let $1 \le i \le n$. $R^{\{i\}}$ is nearly Gorenstein if and only if either $i=1$ or $i=n$ and $\ell_n = 1$.
		\item Let $1 \le i \le n$. $R_{\{i\}}$ is nearly Gorenstein if and only if either $i=n$ or $i=n-1$ and $\ell_n = 1$.
		\item For any $1\le i < j \le n$, $R^{\{i,j\}}$ and $R_{\{i,j\}}$ is not nearly Gorenstein.
		\item Let $1 \le i \le n$ and $1 \le j \le n$. $R^{\{i\}}_{\{j\}}$ is nearly Gorenstein if and only if $n=4$, $i=1$, $j=3$, and $\ell_4 = 1$.
	\end{enumerate}
\end{Theorem}

The ``if'' parts can be proved by direct computation. Here, let us show Theorem \ref{AGcase} (1). We may assume $i=1$. Consider
\begin{eqnarray*}
\mathbf{f}_1 &=& \begin{pmatrix}
	x_1 & x_2 & \cdots &x_{n-1}
\end{pmatrix}, \\
\mathbf{f}_2 &=& \begin{pmatrix}
 	x_2 x_{n-1}^{\ell_{n-1}-1} & x_3 x_{n-1}^{\ell_{n-1}-1} & \cdots &  x_{n-2} x_{n-1}^{\ell_{n-1}-1} & x_{n-1}^{\ell_{n-1}} &  x_n 
\end{pmatrix}, \\
\mathbf{f}_3 &=& \begin{pmatrix}
	x_3 x_{n-1}^{\ell_{n-1}-1} x_n^{\ell_n -1} & x_4 x_{n-1}^{\ell_{n-1}-1} x_n^{\ell_n -1} & \cdots & x_{n-1}^{\ell_{n-1}}x_n^{\ell_n -1} &  x_n^{\ell_n} & x_1 + y_1
\end{pmatrix}
\end{eqnarray*}
in $\Mat_{1 \times n-1}(R^{\{1\}})$.
Then, direct computation shows $\mathbf{f}_i{\cdot}N = 0$ for $i=1,2,3$, hence $x_1, x_2,\ldots , x_n, y_1 \in \trace_{R^{\{1\}}}(\canon_{R^{\{1\}}})$. Therefore, $R^{\{1\}}$ is nearly Gorenstein.
The same computations imply the ``if'' parts of Theorem \ref{AGcase} (2), (3), (4) and Theorem \ref{nonAGcase} (1), (2), (4) by using the following $\mathbf{f}_i$'s:
	
	\medskip
	
	\noindent
\begin{tabular}{|l|l|}
	\hline
	Theorem \ref{AGcase} (2) $i=1$ & $\mathbf{f}_1 = \begin{pmatrix}
 	x_1^{\ell_1}+z_1 & x_2 & x_3 & \cdots & x_{n-1}
 \end{pmatrix}
$\\
	& $\mathbf{f}_2 = \begin{pmatrix}
 	x_2x_{n-1}^{\ell_{n-1}-1} & x_3 x_{n-1}^{\ell_{n-1}-1} & \cdots & x_{n-1}^{\ell_{n-1}} & x_n
 \end{pmatrix}
$\\
	& $\mathbf{f}_3 = \begin{pmatrix}
 	x_3x_{n-1}^{\ell_{n-1}-1}x_n^{\ell_n -1}&x_4x_{n-1}^{\ell_{n-1}-1}x_n^{\ell_n -1} & \cdots & x_{n-1}^{\ell_{n-1}}x_n^{\ell_n -1}  & x_n^{\ell_n} & x_1
 \end{pmatrix}
$\\
	\hline
	Theorem \ref{AGcase} (3) $(i,j)=(1,3)$ & $\mathbf{f}_1 = \begin{pmatrix}
 	x_1^{\ell_1}+z_1 & x_2 & x_3
 \end{pmatrix}
$, $\mathbf{f}_2 = \begin{pmatrix}
 x_3^{\ell_3}+z_3 & x_4 & x_1	
 \end{pmatrix}$
\\
	\hline
	Theorem \ref{AGcase} (4) $(i,j)=(1,3)$ & $\mathbf{f}_1 = \begin{pmatrix}
 	x_1 & x_2 & x_3
 \end{pmatrix}
$, $\mathbf{f}_2 = \begin{pmatrix}
 x_3^{\ell_3}+z_3 & x_4 & x_1+y_1	
 \end{pmatrix}$
\\
	\hline
	Theorem \ref{nonAGcase} (1) $i=1$ & $\mathbf{f}_1 = \begin{pmatrix}
 	x_1 & x_2 & \cdots & x_{n-1}
 \end{pmatrix}
$\\
	& $\mathbf{f}_2 = \begin{pmatrix}
 	x_2x_{n-1}^{\ell_{n-1}-1} & x_3 x_{n-1}^{\ell_{n-1}-1} & \cdots & x_{n-1}^{\ell_{n-1}} & x_n
 \end{pmatrix}
$\\
	& $\mathbf{f}_3 = \begin{pmatrix}
 	x_3x_{n-1}^{\ell_{n-1}-1}x_n^{\ell_n -1}&x_4x_{n-1}^{\ell_{n-1}-1}x_n^{\ell_n -1} & \cdots & x_{n-1}^{\ell_{n-1}}x_n^{\ell_n -1}  & x_n^{\ell_n} & x_1^{m_1}+y_1
 \end{pmatrix}
$\\
	\hline
	Theorem \ref{nonAGcase} (1) $i=n$ & $\mathbf{f}_1 = \begin{pmatrix}
 	x_1 & x_2 & \cdots & x_{n-1}
 \end{pmatrix}
$\\
	& $\mathbf{f}_2 = \begin{pmatrix}
 	x_2x_{n-1}^{\ell_{n-1}-1} & x_3 x_{n-1}^{\ell_{n-1}-1} & \cdots & x_{n-1}^{\ell_{n-1}} & x_n + y_n
 \end{pmatrix}
$\\
	& $\mathbf{f}_3 = \begin{pmatrix}
 	x_n & x_1^{m_1} & x_1^{m_1-1}x_2 & \cdots & x_1^{m_1-1} x_{n-2}
 \end{pmatrix}
$\\
	\hline
	Theorem \ref{nonAGcase} (2) $i=n$ & $\mathbf{f}_1 = \begin{pmatrix}
 	x_1 & x_2 & \cdots & x_{n-1}
 \end{pmatrix}
$\\
	& $\mathbf{f}_2 = \begin{pmatrix}
 	x_2x_{n-1}^{\ell_{n-1}-1} & x_3 x_{n-1}^{\ell_{n-1}-1} & \cdots & x_{n-1}^{\ell_{n-1}} & x_n
 \end{pmatrix}
$\\
	& $\mathbf{f}_3 = \begin{pmatrix}
 	x_n^{\ell_n} + z_n & x_1^{m_1} & x_1^{m_1-1}x_2 & \cdots & x_1^{m_1-1} x_{n-2}
 \end{pmatrix}
$\\
	\hline
	Theorem \ref{nonAGcase} (2) $i=n-1$ & $\mathbf{f}_1 = \begin{pmatrix}
 	x_1 & x_2 & \cdots & x_{n-1}
 \end{pmatrix}
$\\
	& $\mathbf{f}_2 = \begin{pmatrix}
	x_{n-1}^{\ell_{n-1}} + z_{n-1} & x_n & x_1^{m_1} & x_1^{m_1 -1}x_2 & \cdots & x_1^{m_1-1} x_{n-3}
\end{pmatrix}
$\\
	\hline
	Theorem \ref{nonAGcase} (4) & $\mathbf{f}_1 = \begin{pmatrix}
 	x_1 & x_2 & x_3
 \end{pmatrix}
$, $\mathbf{f}_2 = \begin{pmatrix}
 	x_3^{\ell_3} + z_3 & x_4 & x_1^{m_1} + y_1
 \end{pmatrix}$
\\
	\hline
\end{tabular}

\subsection{Proof of Theorems \ref{AGcase} and \ref{nonAGcase}}\label{subsection4.4}

Throughout this subsection, suppose that $n \ge 4$.

Let $$B^I_J = k[\{X_i\}_{1\le i \le n}, \{Y_{i_\alpha}\}_{1\le \alpha \le p}, \{Z_{j_\beta}\}_{1\le \beta \le q}]$$
 be the polynomial ring over $k$ with $n+p+q$ indeterminates.

We regard $B^I_J$ as a graded ring by $[B^I_J]_0 = k$, $\deg X_i = a_i$, $\deg Y_{i_\alpha} = m_{i_\alpha} a_{i_\alpha}$, and $\deg Z_{j_\beta} = \ell_{j_\beta} a_{j_\beta}$ for $1 \le i \le n$, $1\le \alpha \le p$, and $1 \le \beta \le q$.
With this grading, $J^I_J =\detid_2(D^I_J)$ becomes a graded ideal of $B^I_J$. Consequently, the quotient ring $A^I_J = B^I_J / J^I_J$ is also graded.
Since $R^I_J = \widehat{[A^I_J]_\fkM}$ where $\fkM$ denotes the graded maximal ideal of $A^I_J$ and $\widehat{*}$ denotes the $\fkM[A^I_J]_\fkM$-adic completion, by the same way as the argument after Proposition \ref{3.1}, we can use the graded information to analyze the nearly Gorenstein property on $R^I_J$.
In what follows, we say that an element $\varphi \in R^I_J$ is homogeneous if $\varphi$ is an image of some homogeneous element in $A^I_J$.

We put $$\Lambda = \{\mathbf{f} = \begin{pmatrix}f_1 & f_2 & \cdots &f_{n-1} \end{pmatrix} \in \Mat_{1 \times n-1}(R^I_J) \mid \mathbf{f}{\cdot}N = 0 \text{ in $\Mat_{1 \times n(n-2)}(R^I_J)$}\}.$$
For $x\in R$ and $\mathbf{f} = \begin{pmatrix}f_1 & f_2 & \cdots &f_{n-1} \end{pmatrix} \in \Lambda$, $x\in \mathbf{f}$ denotes $x=f_r$ for some $1 \le r \le n-1$. 
We begin with the following.

\begin{Lemma}\label{lem1}
We have the following.
\begin{enumerate}[{\rm (1)}]
	\item Let $1 \le i \le n$ be an integer and $\varphi \in R^I_J$ a homogeneous element.
	If $\varphi v_{\nint{i+1}} \in (u_i)$, then $\varphi \in (u_1, u_2,\ldots , u_n)$.
	\item Let $1 \le j \le n$ be an integer and $\varphi \in R^I_J$ a homogeneous element.
	If $\varphi u_{i} \in (v_{\nint{i+1}})$, then $\varphi \in (v_1, v_2,\ldots , v_n)$.
	\item Let $1 \le i \le n$ such that $i \notin J$, and $\varphi \in R^I_J$ a homogeneous element.
If $\varphi v_{[i+1]} \in (x_i)$, then $\varphi \in (\{u_1, u_2,\ldots , u_n\}\setminus\{u_i\}) + (x_i)$.
	\item Let $1 \le i \le n$ such that $[i+1] \notin I$, and $\varphi \in R^I_J$ a homogeneous element.
If $\varphi u_{i} \in (x_{[i+1]})$, then $\varphi \in (\{v_1, v_2, \ldots , v_n\} \setminus \{v_{[i+1]}\}) + (x_{[i+1]})$.
\end{enumerate}
\end{Lemma}

\begin{proof}
	We only prove (1) since (2), (3), and (4) can be proved in the same manner as (1). 
	Suppose $\varphi v_{[i+1]} \in (u_i)$. We may assume $i=1$. If $2\in I$ or $1\in J$, then $u_1$, $v_2$ is an $R^I_J$-regular sequence. It follows that $\varphi\in (u_1)$. Thus we may assume $2\notin I$ and $1\notin J$.  
	We choose $f \in S^I_J$ such that $\varphi = \overline{f}$ in $R^I_J$. Then, 
	\begin{eqnarray*}
	fV_2 &\in& (U_1) + \detid_2\begin{pmatrix}V_2 & V_3 & \cdots & V_n & V_1\\0 & U_2 & \cdots & U_{n-1} & U_n\end{pmatrix}\\
	&=&	(U_1) + V_2 (U_2,U_3,\ldots ,U_n) + \detid_2(D'),
	\end{eqnarray*}
	where $D'=\begin{pmatrix}
		V_3 & \cdots & V_n & V_1\\
		U_2 & \cdots & U_{n-1} & U_n
	\end{pmatrix}$. Therefore, there exists $g \in (U_2, \ldots , U_n)$ such that $(f-g)V_2 \in (U_1)+\detid_2(D')$. 
	We note that $X_1, X_2$ is an $S^I_J/\detid_2(D')$-regular sequence. 
	Therefore, since $U_1=X_1^{m_1}$ and $V_2=X_2^{\ell_2}$ by $2\notin I$ and $1\notin J$, we have $f-g\in (U_1)+\detid_2(D')$. It follows that $\varphi = \overline{f} \in (u_1, u_2,\ldots , u_n)$.
\end{proof}

\begin{Lemma}\label{lem2}
	We have the following.
	\begin{enumerate}[{\rm (1)}]
		\item 	Suppose $p > 0$. Let $i \in I$ and $\varphi \in R^I_J \setminus (u_1,u_2,\ldots , u_n)$ a homogeneous element.  
	If there exists $\mathbf{f} = \begin{pmatrix}f_1 & f_2 & \cdots &f_{n-1}\end{pmatrix} \in \Lambda$ such that $v_i \varphi \in \mathbf{f}$, then $f_{n-1} = v_i \varphi$.
		\item Suppose $q > 0$. Let $j \in J$ and $\varphi \in R^I_J \setminus (v_1,v_2,\ldots , v_n)$ a homogeneous element.  
	If there exists $\mathbf{f} = \begin{pmatrix}f_1 & f_2 & \cdots &f_{n-1}\end{pmatrix} \in \Lambda$ such that $u_j \varphi \in \mathbf{f}$, then $f_{1} = u_j \varphi$.
	\end{enumerate}
\end{Lemma}

\begin{proof}
	We will focus only on the proof of (1) since the proof of (2) follows from the same procedure as that of (1). 
	We may assume $i=n$. Let $1 \le \alpha \le n-1$ such that $f_\alpha = v_n \varphi$.
	Suppose that $\alpha < n-1$. Then we have
	$$
	f_\alpha v_1 - f_{\alpha +1} u_n = v_n \varphi v_1 - f_{\alpha+1} u_n = 0.
	$$
	Since $v_n = x_n^{m_n} + y_n, u_n$ forms an $R^I_J$-regular sequence (Remark \ref{new4.0} (1)), $\varphi v_1 \in (u_n)$. This is a contradiction because of Lemma \ref{lem1}.
	Hence $\alpha =n-1$.
\end{proof}

\begin{Proposition}\label{propp>0}
	Suppose $p >0$. Let $i \in I$.
	If there exists $\mathbf{f} = \begin{pmatrix}f_1 & f_2 & \cdots &f_{n-1}\end{pmatrix} \in \Lambda$ such that $v_i = x_i^{m_i}+y_i \in \mathbf{f}$, then we have the following.
	\begin{enumerate}[{\rm (1)}]
		\item $\nint{i-1} \notin I \cup J$.
		\item $m_{\nint{i-1}} \le \ell_{\nint{i-1}}$.
		\item $f_{n-1} = x_i^{m_i} + y_i$, $f_{n-2} = u_{\nint{i-1}}=x_{\nint{i-1}}^{\ell_{\nint{i-1}}}$, and $f_{n-3} = u_{\nint{i-2}} x_{\nint{i-1}}^{\ell_{\nint{i-1}} - m_{\nint{i-1}}}$.
	\end{enumerate}
	If we assume $n \ge 5$ and $\ell_\alpha < 2m_\alpha$ for all $\alpha \in \{\nint{i-1}, \nint{i-2}, \ldots , \nint{i-r}\}$ $(1 \le r \le n-4)$, then we further have the following.
	
	\begin{enumerate}[{\rm (1)}]\setcounter{enumi}{3}
		\item $\alpha \notin I\cup J$ for all $\alpha \in \{\nint{i-2}, \nint{i-3},\ldots , \nint{i-r-1}\}$.
		\item $m_\alpha \le \ell_\alpha$ for all $\alpha \in \{\nint{i-2}, \nint{i-3},\ldots , \nint{i-r-1}\}$.
		\item $f_k = u_{\nint{i-n+r+1}}{\cdot}\prod_{\alpha =1}^{n-r-2} x_{\nint{i-\alpha}}^{\ell_{\nint{i-\alpha}} - m_{\nint{i-\alpha}}}$ for all $n-3-r \le k \le n-4$.
	\end{enumerate}
\end{Proposition}

\begin{proof}
	We may assume $i=n$. By applying Lemma \ref{lem2} (1) with $\varphi=1$, $f_{n-1} = v_n = x_n^{m_n} + y_n$.
	Then, since $$f_{n-2} v_n - f_{n-1} u_{n-1} = f_{n-2} v_n - v_n u_{n-1} = 0,$$
	we have $f_{n-2} = u_{n-1}$.
	Suppose that $n-1\in J$. Then, by Lemma \ref{lem2} (2), $f_1=u_{n-1}=f_{n-2}$. By Remark \ref{new2.25}, we have $n-2=1$, which is a contradiction to the assumption that $n \ge 4$. Hence $n-1 \notin J$. 
	Assume $n-1 \in I$. 
	Then $$f_{n-3} v_{n-1} - f_{n-2} u_{n-2} = f_{n-3} v_{n-1} - x_{n-1}^{\ell_{n-1}} u_{n-2} = 0.$$
	Hence $u_{n-2} \in (v_{n-1})$ because $v_{n-1}, x_{n-1}^{\ell_{n-1}}$ is an $R^I_J$-regular sequence (Remark \ref{new4.0} (1)). 
	However, this is impossible by Remark \ref{new4.0} (2). Thus, $n-1 \notin I$.
	
	By (1), $v_{n-1} = x_{n-1}^{m_{n-1}}$. It follows that $v_{n-2}\notin (x_{n-1})$ by Remark \ref{new4.0} (2). This implies $m_{n-1} \le \ell_{n-1}$ by assuming the contrary and considering the equality $$f_{n-3} x_{n-1}^{m_{n-1}} - x_{n-1}^{\ell_{n-1}} u_{n-2} = 0.$$
	So, we get $f_{n-3} = u_{n-2} x_{n-1}^{\ell_{n-1} -m_{n-1}}$ and complete proofs of (1), (2), and (3).
	
	Suppose that $n\ge 5$ and $\ell_\alpha < 2m_\alpha$ for all $4 \le \alpha \le n-1$. We then prove that (1), (2), and (3) imply that (4), (5), and (6) hold in the case of $\alpha=i-2 = n-2$. 
	
	Suppose $n-2 \in J$. We note that $f_{n-3} = u_{n-2}x_{n-1}^{\ell_{n-1} -m_{n-1}}$ and $x_{n-1}^{\ell_{n-1} - m_{n-1}} \notin (v_1, \ldots, v_n)$ since $\ell_{n-1} - m_{n-1}<m_{n-1}$. Hence, by Lemma \ref{lem2} (2), $f_{n-3} = u_{n-2}x_{n-1}^{\ell_{n-1} - m_{n-1}} =f_1$. 
	By Remark \ref{new2.25}, this implies $n-3=1$, a contradiction. Thus $n-2\notin J$.

	Suppose that $n-2 \in I$. 
	Then 
	\begin{align}\label{eq3}
	f_{n-4} v_{n-2} - f_{n-3}u_{n-3} = f_{n-4}v_{n-2} - x_{n-2}^{\ell_{n-2}} x_{n-1}^{\ell_{n-1} -m_{n-1}}u_{n-3} = 0
	\end{align}
	by (3). 
	Since $v_{n-2}, x_{n-2}^{\ell_{n-2}}$ is an $R^I_J$-regular sequence, $x_{n-1}^{\ell_{n-1} - m_{n-1}} u_{n-3} \in (v_{n-2})$. 
	However, this contradicts to $x_{n-1}^{\ell_{n-1} - m_{n-1}}\not \in (v_1, \ldots, v_n)$ by applying Lemma \ref{lem1} (2). 
	Hence $n-2 \notin I$. 
	Suppose that $m_{n-2} > \ell_{n-2}$. Then, by the equality \eqref{eq3}, we have $x_{n-1}^{\ell_{n-1} - m_{n-1}} u_{n-3} =f_{n-4} x_{n-2}^{m_{n-2}-\ell_{n-2}} \in (x_{n-2})$.
	Therefore $x_{n-1}^{\ell_{n-1} -m_{n-1}} \in (u_1, u_2,\ldots u_{n-3}, u_{n-1}, u_n) + (x_2)$ by Lemma \ref{lem1} (4). But this is a contradiction because $\ell_{n-1} < 2m_{n-1}$.
	It follows that by \eqref{eq3} that $f_{n-4} = x_{n-2}^{\ell_{n-2} - m_{n-2}} x_{n-1}^{\ell_{n-1} -m_{n-1}}u_{n-3}$. Therefore, assertions (4), (5), and (6) in the case of $\alpha=n-2$ hold. Then we can complete the proof inductively.	
\end{proof}

By arranging the form of the matrix $D^I_J$, the next proposition follows from Proposition \ref{propp>0}.

\begin{Proposition}\label{propq>0}
	Suppose $q >0$. Let $j \in J$.
	If there exists $\mathbf{f} = \begin{pmatrix}f_1 & f_2 & \cdots &f_{n-1}\end{pmatrix} \in \Lambda$ such that $u_j = x_j^{\ell_j}+z_j \in \mathbf{f}$, then we have the following.
	\begin{enumerate}[{\rm (1)}]
		\item $\nint{j+1} \notin I \cup J$.
		\item $\ell_{\nint{j+1}} \le m_{\nint{j+1}}$.
		\item $f_{1} = x_j^{\ell_j} + z_j$, $f_{2} = v_{\nint{j+1}}=x_{\nint{j+1}}^{m_{\nint{j+1}}}$, and $f_{3} = v_{\nint{j+2}} x_{\nint{j+1}}^{m_{\nint{j+1}} - \ell_{\nint{j+1}}}$.
	\end{enumerate}
	If we assume $n \ge 5$ and $m_\alpha < 2\ell_\alpha$ for all $\alpha \in \{\nint{j+1}, \nint{j+2}, \ldots , \nint{j+r}\}$ $(1 \le r \le n-4)$, then we further have the following.
	
	\begin{enumerate}[{\rm (1)}]\setcounter{enumi}{3}
		\item $\alpha \notin I\cup J$ for all $\alpha \in \{\nint{j+2}, \nint{j+3},\ldots , \nint{j+r+1}\}$.
		\item $\ell_\alpha \le m_\alpha$ for all $\alpha \in \{\nint{j+2}, \nint{j+3},\ldots , \nint{j+r+1}\}$.
		\item $f_k = u_{\nint{j+r-1}}{\cdot}\prod_{\alpha =1}^{r-2} x_{\nint{j+\alpha}}^{\ell_{\nint{j+\alpha}} - m_{\nint{j+\alpha}}}$ for all $4 \le k \le r+3$.
	\end{enumerate}
\end{Proposition}

%


\begin{Proposition}\label{prop2p>0}
	Let $1\ \le i \le n$. Suppose that $i \in I$ or $m_i \ge 2$. 
	If there exists $\mathbf{f} \in \Lambda$ such that $x_i \in \mathbf{f}$, then we have the following.
	\begin{enumerate}[{\rm (1)}]
		\item $i \notin J$ and $\ell_i = 1$.
		\item $\nint{i+1} \notin I \cup J$ and $\ell_{\nint{i+1}} \le m_{\nint{i+1}}$.
		\item $f_1 = x_i$, $f_2 = x_{\nint{i+1}}^{m_{\nint{i+1}}}$, and $f_3 = v_{\nint{i+2}} x_{\nint{i+1}}^{m_{\nint{i+1}}-\ell_{\nint{i+1}}}$.
	\end{enumerate}
	If we assume $n \ge 5$ and $m_\alpha < 2\ell_\alpha$ for all $\alpha \in \{\nint{i+1}, \nint{i+2},\ldots , \nint{i+r}\}$ $(1 \le r \le n-4)$, then we further have the following.
	\begin{enumerate}[{\rm (1)}] \setcounter{enumi}{3}
		\item $\alpha \notin I \cup J$ for all $\alpha \in \{\nint{i+2}, \nint{i+3},\ldots , \nint{i+r+1}\}$.
		\item $m_{\alpha} \ge \ell_{\alpha}$ for all $\alpha \in \{\nint{i+2}, \nint{i+3},\ldots , \nint{i+r+1}\}$.
		\item $f_k = v_{\nint{i+r-1}} \prod_{\alpha=1}^{r-2} x_{\nint{i+\alpha}}^{m_{\nint{i+\alpha}} - \ell_{\nint{i+\alpha}}}$ for all $4 \le k \le r+3$.
	\end{enumerate}
\end{Proposition}

\begin{proof}
	We may assume $i=1$. 
	First, suppose that $f_\alpha = x_1$ for $2 \le \alpha \le n-1$.
	Then $f_{\alpha-1} v_1 - x_1 u_n = 0$. Noting $u_n \notin (v_1)$ by Remark \ref{new4.0} (2), we have $1 \notin I$. Hence $m_1 \ge 2$. 
	Therefore $$f_{\alpha-1} x_1^{m_1} - x_1 u_n = x_1(f_{\alpha-1} x_1^{m_1-1} - u_n) = 0.$$ 
	Thus $u_n \in (x_1^{m_1-1})$ which is impossible by Remark \ref{new4.0} (2), whence $f_1 = x_1$.
	
	Now, suppose that $1 \in J$. 
	Then, since $x_1 v_2 - f_2 u_1 = 0$ and $x_1, u_1$ is an $R^I_J$-regular sequence, we have $v_2 \in (u_1)$, a contradiction.
	Hence $1 \notin J$. 
	On the other hand, if $\ell_1 \ge 2$, then the same equality $x_1v_2 - f_2 u_1 = 0$ implies that $v_2 \in (x_1^{\ell_1-1})$ but this is also a contradiction. 
	Therefore $\ell_1 =1$. Thus $u_1=x_1$, whence we get $f_2 = v_2$ by the equality $x_1v_2 - f_2 u_1 = 0$. 
	
	Next, we assume $2 \in I$. Then $f_2 = v_2 = x_2^{m_2} + y_2$. 
	Hence by Lemma \ref{lem2} (1) and Remark \ref{new2.25}, we have $2 = n-1$, which is a contradiction. 
	Therefore $2 \notin I$ and $f_2 = v_2 = x_2^{m_2}$.	

	Now, assume that $2 \in J$. 
	Then $$f_2 v_3 - f_3 u_2 = x_2^{m_2} v_3 - f_3 (x_2^{\ell_2} + z_2) =0.$$
	Hence the fact that $x_2^{m_2}, u_2 = x_2^{\ell_2} + z_2$ is an $R^I_J$-regular sequence leads to a contradiction that $v_3 \in (u_2)$. 
	Therefore $2 \notin J$ and $x_2^{m_2} v_3 - f_3x_2^{\ell_2} = 0$. If $m_2 < \ell_2$, then $v_3 \in (x_2^{\ell_2-m_2}) \subseteq (x_2)$ which is impossible. 
	Hence $m_2 \ge \ell_2$ and $f_3 = v_3 x_2^{m_2 -\ell_2}$. Thus, we complete proofs of (1), (2), and (3). 
	
	Suppose $n\ge 5$ and $m_\alpha < 2\ell_\alpha$ for all $2 \le \alpha \le n-3$. We then prove that (1), (2), and (3) imply that (4), (5), and (6) hold in the case of $\alpha=i+2=3$. 
	To prove (4), assume $3 \in I$. Then $f_3 = v_3 x_2^{m_2 -\ell_2}$ implies $3 =n-1$ by Lemma \ref{lem2} (1) and Remark \ref{new2.25}. Thus $n=4$, but this contradicts our assumption. Hence $3 \notin I$.
	To prove (5) and (6), suppose that $3 \in J$. 
	Then $$f_3 v_4 - f_4 u_3 = (x_3^{m_3} x_2^{m_2-\ell_2}) v_4 - f_4 (x_3^{\ell_3} + z_3) = 0.$$ 
	Since $x_3^{m_3}, u_3=x_3^{\ell_3}+z_3$ is an $R^I_J$-regular sequence, we have $x_2^{m_2-\ell_2}v_4 \in (u_3)$ which contradicts to Lemma \ref{lem1} (1), because $m_2 - \ell_2 < \ell_2$. 
	Hence $3 \notin J$ and $f_4 = v_4 x_3^{m_3-\ell_3} x_2^{m_2-\ell_2}$. Thus, (4), (5), and (6) hold in the case of $\alpha=3$.
	Then assertions (4), (5), and (6) can be shown inductively.
\end{proof}

By arranging the form of the matrix $D^I_J$, the next proposition follows from Proposition \ref{prop2p>0}.

\begin{Proposition}\label{prop2q>0}
	Let $1\ \le j \le n$. Suppose that $j \in J$ or $\ell_j \ge 2$. 
	If there exists $\mathbf{f} \in \Lambda$ such that $x_j \in \mathbf{f}$, then we have the following.
	\begin{enumerate}[{\rm (1)}]
		\item $j \notin I$ and $m_j = 1$.
		\item $\nint{j-1} \notin I \cup J$ and $m_{\nint{j-1}} \le \ell_{\nint{j-1}}$.
		\item $f_{n-1} = x_j$, $f_{n-2} = x_{\nint{j-1}}^{\ell_{\nint{j-1}}}$, and $f_{n-3} = u_{\nint{j-2}}x_{\nint{j-1}}^{m_{\nint{j-1}} - \ell_{\nint{j-1}}}$.
	\end{enumerate}
	If we assume $n \ge 5$ and $\ell_\alpha < 2m_\alpha$ for all $\alpha \in \{\nint{j-1}, \nint{j-2},\ldots , \nint{j-r}\}$ $(1\le r \le n-4)$, the we furthermore have the following.
	\begin{enumerate}[{\rm (1)}] \setcounter{enumi}{3}
		\item $\alpha \notin I \cup J$ for all $\alpha \in \{\nint{j-2}, \nint{j-3},\ldots , \nint{j-r-1}\}$.
		\item $m_{\alpha} \ge \ell_{\alpha}$ for all $\alpha \in \{\nint{j-2}, \nint{j-3},\ldots , \nint{j-r-1}\}$.
		\item $f_r = v_{\nint{j-n+r+1}} \prod_{\alpha=1}^{n-r-2} x_{\nint{j-\alpha}}^{\ell_{\nint{j-\alpha}} - m_{\nint{j-\alpha}}}$ for all $n-3-r \le r \le n-4$.
	\end{enumerate}
\end{Proposition}

\begin{Lemma}\label{lambda}
 	Suppose that $R^I_J$ is nearly Gorenstein.
	We then have the following.
	\begin{enumerate}[{\rm (1)}]
		\item $I\cap J = \emptyset$.
		\item For all $i \in I$, there exist $\mathbf{f}, \mathbf{g} \in \Lambda$ such that  $v_i = x_i^{m_i} + y_i \in \mathbf{f}$ and $x_i\in \mathbf{g}$.
		\item For all $j \in J$, there exist $\mathbf{f}, \mathbf{g} \in \Lambda$ such that $v_i = x_j^{\ell_j} + z_j \in \mathbf{f}$ and $x_j\in \mathbf{g}$.
	\end{enumerate}
\end{Lemma}

\begin{proof}
	Notice that $[R^I_J]_h \subseteq [\trace_{R^I_J}(\canon_{R^I_J})]_h$ for all $0<h \in \mathbb{N}$ because $R^I_J$ is nearly Gorenstein. 

	(1) Suppose that $I \cap J \ne \emptyset$ and let $i \in I \cap J$. We may assume $i=1$.
	By Proposition \ref{prop2p>0} (1), $x_1 \notin \mathbf{f}$ for all $\mathbf{f} \in \Lambda$. It follows that $\dim_k [R^I_J]_{a_1} \ge 2$ by Corollary \ref{graded dim1}. On the other hand, by noting that $a_1, \ldots, a_n$ minimally generates $H$, possible monomials of degree $a_1$ are only $X_1$, $Y_1$, and $Z_1$ in $T^I_J$ (see the beginning of Subsection \ref{subsection4.4}). Since the degrees of the latter two monomials are $m_1 a_1$ and $\ell_1 a_1$ respectively, we have $m_1 = 1$ or $\ell_1 = 1$. 
	
	First, assume that $m_1 = \ell_1 =  1$. Hence, $\dim_k [R^I_J]_{a_1} =3$. Put $\varphi = c x_1 + d y_1 + e z_1$ $(c,d,e \in k)$.
	Suppose that $0 \ne \varphi \in \mathbf{f}$ for some $\mathbf{f} \in \Lambda$.
	If $f_\alpha = \varphi$ for $2 \le \alpha \le n-1$, then 
	$$
	f_{\alpha -1} (x_1 + y_1) - \varphi u_n = f_{\alpha -1} (x_1 +y_1) - (cx_1 + dy_1 + ez_1) u_n = x_1(f_{\alpha-1}-cu_n) + y_1(f_{\alpha-1}-du_n) -eu_nz_1 = 0.$$
	Therefore $eu_n \in (x_1, y_1)$ because $x_1, y_1, z_1$ is an $R^I_J$-regular sequence, and hence $e=0$.
	We then have $f-cu_n \in (y_1)$ and $f-du_n \in (x_1)$. Thus $(c-d)u_n \in (x_1, y_1)$ which implies $c-d =0$. Thus $\varphi = c(x_1 + y_1)$ if $f_\alpha = \varphi$ for $2 \le \alpha \le n-1$.
	On the other hand, if $f_\alpha = \varphi$ for $1 \le \alpha \le n-2$, then the same argument leads to $\varphi = c(x_1 + z_1)$. 
	Hence, possible forms of $\varphi\in \mathbf{f}$ for some $\mathbf{f}\in \Lambda$ are in $k(x_1+y_1) + k(x_1+z_1)$. It follows that $[R^I_J]_{a_1} = k(x_1+y_1) + k(x_1+z_1)$ by Proposition \ref{graded}. This is a contradiction to $\dim_k  [R^I_J]_{a_1} = 3$. Hence either $m_1=1$ or $\ell_1=1$ holds.  Thus, we may assume $m_1 = 1$ and $\ell_1 \ge 2$.

	Suppose $\varphi = cx_1 + dy_1 \in \mathbf{f}$ for some $c,d \in k$ and $\mathbf{f} \in \Lambda$.
	If $f_\alpha = \varphi$ for $2 \le \alpha \le n-1$, we yield $c=d$ by the same argument as above.
	Suppose $f_\alpha = \varphi$ for $1\le \alpha \le n-2$. 
	Then $$\varphi v_2 -f_{\alpha+1} u_1 = (cx_1 + dy_1)v_2 - f_{\alpha+1}(x_1^{\ell_1} + z_1) = x_1(cv_2 -x_1^{\ell_1 - 1}) + y_1(dv_2) - f_{\alpha+1}z_1 = 0.$$
	Hence $dv_2 \in (x_1, z_1)$, because $x_1,z_1,y_1$ is an $R^I_J$-regular sequence.
	Thus $d=0$, but this contradicts the fact that $x_1 \notin \mathbf{f}$ for every $\mathbf{f} \in \Lambda$ as we have seen above.
	
	(2) We may assume $i=1$. By (1), we have $1 \notin J$.
	Since $R^I_J$ is nearly Gorenstein and $\dim_k [R^I_J]_{m_1 a_1} =2$, there should exist $c, d\in k$ such that $cx_1^{m_1} + dy_1 \in \mathbf{f}$ for some $\mathbf{f} \in \Lambda$.
	By the same argument as above, if $d \ne 0$, then we have $c=d$. Hence $x_1^{m_1} + y_1 \in \mathbf{f}$ for some $\mathbf{f}\in \Lambda$.
	
	If $m_1 \ge 2$, then $\dim_k [R^I_J]_{a_1} = 1$. Hence, by Corollary \ref{graded dim1} and $R^I_J$ is nearly Gorenstein, $x_1 \in \mathbf{g}$ for some $\mathbf{g} \in \Lambda$. So, the proof in this case is done.
	We now suppose $m_1 =1$.
	Then $\dim_k [R^I_J]_{a_1} = 2$. Hence, by Corollary \ref{graded dim2}, there exist $c,d \in k$ such that $c \ne d$ and $cx_1 + dy_1 \in \mathbf{g} $ for some $\mathbf{g}= \begin{pmatrix}g_1 & g_2 & \cdots & g_{n-1}\end{pmatrix} \in \Lambda$, because we have already shown that $x_1 + y_1 \in \mathbf{f}$.
	By the same argument as in the proof of (1), the condition $c \ne d$ implies that $g_1 = cx_1 + dy_1$.
	Hence $$(cx_1 + dy_1) v_2 - g_2 x_1^{\ell_1} = x_1 ( cv_2 - g_2x_1^{\ell_1-1}) + y_1 (dv_2) = 0.$$ 
	Hence $d=0$ because $v_2 \notin (x_1)$. Therefore $c \ne 0$ and $cx_1 \in \mathbf{g}$. 
	We can also prove (3) in a similar way.
\end{proof}

Notice that, by Lemma \ref{lambda}, in each case in Theorems \ref{AGcase} and \ref{nonAGcase}, if $R^I_J$ is nearly Gorenstein, then there exist $\mathbf{f}_i, \mathbf{g}_i \in \Lambda$ such that $x_i \in \mathbf{f}_i$ and $v_i \in \mathbf{g}_i$ for all $i \in I$ and there exist $\mathbf{f}'_j, \mathbf{g}'_j \in \Lambda$ such that $x_j \in \mathbf{f}'_j$ and $u_j \in \mathbf{g}'_j$ for all $j \in J$.
Then we are now in the position to give proofs of Theorems \ref{AGcase} and \ref{nonAGcase}.

\begin{proof}[Proof of Theorem \ref{AGcase}]
	First, notice that $m_\alpha = 1 < 2\ell_\alpha$ for all $1\le \alpha \le n$. 

	(1) Suppose that $R^{\{i\}}$ is nearly Gorenstein. We may assume $i=1$. Then, by Proposition \ref{prop2p>0} (2) and (5), $\ell_1=1$ and $\ell_\alpha \le m_\alpha = 1$ for all $2 \le \alpha \le n-2$. 
	
	\medskip
	
	(2) Suppose that $R_{\{i\}}$ is nearly Gorenstein. We may assume $i=1$. 
	Then, by Proposition \ref{propq>0} (2) and (5), $\ell_\alpha \le m_\alpha = 1$ for all $2 \le \alpha \le n-2$.

	\medskip

	(3) Suppose that $R_{\{i,j\}}$ is nearly Gorenstein. We may assume $i=1$.
		The assertion (2) and Lemma \ref{new4.5} show that $\ell_2 = \ell_3 = \cdots = \ell_{n-2} = 1$. 
		By Proposition \ref{propq>0} (1) and (4), we get $2,\ldots, n-2 \not\in J=\{i, j\}$; hence $j=n-1$ or $j=n$.
		
		If $j=n$, then $\nint{j+1} = 1 \notin J=\{1,n\}$ by Proposition \ref{propq>0} (1), but this is impossible.
		Hence $j=n-1$. In this case, if $n \ge 5$, then $1 =\nint{j+2} \in \{\nint{j+1}, \ldots , \nint{j+n-3}\}$, which is also impossible by Proposition \ref{propq>0} (4). Hence $n=4$ and $j=n-1=3$. 	

	\medskip
	
	(4) Suppose $R^{\{i\}}_{\{j\}}$ is nearly Gorenstein. We may assume $i=1$.
		By Lemma \ref{lambda} (1), we have $j \ne 1$. 
		On the other hand, Proposition \ref{prop2p>0} (2) and (4) leads $j \notin \{2, 3, \ldots, n-2\}$.
		Hence $j=n-1$ or $j=n$. 
		If $j=n$, then we get a contradiction that $1 \notin \{1,2,\ldots , n-3\}$ by Proposition \ref{propq>0} (1) and (4). Hence $j = n-1$.
		Moreover, if $n \ge 5$, then we again get a contradiction that $1 =\nint{n+1} \notin \{n, \nint{n+1}, \ldots, \nint{2n-3}\}$ by Proposition \ref{propq>0} (1) and (4). Hence $n=4$ and $j=n-1= 3$.
		By Lemma \ref{new4.5}, assertions (1) and (2) imply $\ell_1=\ell_2 = \ell_4=1$.
	
	\medskip

	(5) Suppose $R^{\{i,j\}}$ is nearly Gorenstein. We may assume $i=1$. 
	By Proposition \ref{prop2p>0} (2) and (4), we have $2,\ldots , n-2 \notin I$. Hence $j=n-1$ or $j=n$.
	However, in each case, we have $1 \notin I = \{i,j\}$ by Proposition \ref{prop2p>0} (2) and (4).

	\medskip
	
	(6) Suppose $R^{\{k\}}_{\{i,j\}}$ is nearly Gorenstein, then $R^{\{k\}}_{\{i\}}$ and $R^{\{k\}}_{\{j\}}$ are also nearly Gorenstein by Lemma \ref{new4.5}.
	By (4), we have $n=4$ and $(k,i), (k,j) \in \{(1,3), (2,4), (3,1), (4,2)\}$. Then $i=j$, but this is impossible.

	\medskip
	
	(7) Suppose $R_{\{i,j,k\}}$ is nearly Gorenstein. Then, by Lemma \ref{new4.5} and (3), $n=4$ and $\nint{i-j} = \nint{j-k} = \nint{k-i} =2$. This is impossible. 
\end{proof}

\begin{proof}[Proof of Theorem \ref{nonAGcase}]
	Notice that $m_i =1 < 2\ell_i$ for all $2 \le i \le n$ and $\ell_i =1 < 2m_i$ for all $1 \le i \le n-2$, in this case.

	(1) Suppose $R^{\{i\}}$ is nearly Gorenstein. Assume $2 \le i \le n-1$. Then, $\ell_\alpha< 2 m_\alpha$ for all $1 \le \alpha \le n-2$.
	By Proposition \ref{propp>0} (2) and (5), $m_1 \le \ell_1$. This is a 	contradiction to the assumption. Thus, $i=1$ or $i=n$. 
	Furthermore, by Proposition \ref{prop2p>0} (1), $\ell_n = 1$ when $i=n$.

	\medskip
	
	(2) Suppose $R_{\{i\}}$ is nearly Gorenstein. 
	By Proposition \ref{prop2q>0} (1) and (2), we have $i \ne 1, 2$. In fact, if $i=1$, then we obtain a contradiction that $m_1 = 1$, and if $i=2$, then we get $m_1 \le \ell_1 = 1$, whence $m_1=1$.
	Assume $3 \le i \le n-2$. Since $m_i \le 2\ell_i$ for all $2\le i \le n$, we can apply Proposition \ref{propq>0} with $r=n-4-(i-3)$ if $n\ge 5$. Hence, by Proposition \ref{propq>0} (2) and (5), $\ell_\alpha \le m_\alpha$ for all 
	$i+1 \le \alpha \le i+r+1=n$. In particular, $\ell_{n-1} \le m_{n-1}=1$ and $\ell_n=m_n=1$. This is a contradiction to the assumption that $\ell_{n-1} \ge 2$ or $\ell_n \ge 2$.
	Therefore, $i=n-1$ or $i=n$.
	Moreover, if $i=n-1$, then $\ell_n = 1$ by Proposition \ref{propq>0} (2).

	\medskip
	
	(3) Suppose $R^{\{i,j\}}$ is nearly Gorenstein.
	Then $R^{\{i\}}$ and $R^{\{j\}}$ are nearly Gorenstein  by Lemma \ref{new4.5}.
	Therefore $i=1$, $j=n$, and $\ell_n = 1$ by (1). 
	However, by Proposition \ref{propp>0} (1), we know that $\nint{i-1} = n \notin I$, which is a contradiction.
	Thus $R^{\{i,j\}}$ is not nearly Gorenstein.

	\medskip

	(4) Suppose $R^{\{i\}}_{\{j\}}$ is nearly Gorenstein.
	Since $R^{\{i\}}$ and $R_{\{j\}}$ are nearly Gorenstein by Lemma \ref{new4.5}, we know by (1) and (2) that ``$i=1$ or $i=n$'' and ``$j=n-1$ or $j=n$''.
	Suppose $i=n$. Then $\nint{i-1} = n-1 \notin J$ by Proposition \ref{propp>0} (1). Hence $j=n$, but this contradicts Lemma \ref{lambda} (1).
	Hence $i=1$. Then, again by Proposition \ref{propp>0} (1), $\nint{i-1} = n \notin J$. Thus we yield $j = n-1$ and $\ell_n = 1$.
	However, if $n \ge 5$, then, by applying Proposition \ref{propp>0} (4) in the case of $r=1$, we have $\nint{1-2} = n-1 \notin J$, which is a contradiction.
	Therefore $n=4$, $i=1$, $j=n-1=3$, and $\ell_n = 1$ as desired.
\end{proof}

\begin{Remark}
	According to Propositions \ref{n3AGL} and \ref{n3nonAGL} and Theorems \ref{newAGcase} and \ref{newnonAGcase}, the dimension of the nearly Gorenstein ring obtained by the method in this section is 
	\begin{itemize}
		\item at most 4, if $n=3$,
		\item at most 3, if $n=4$, and
		\item at most 2, if $n\ge 5$.
	\end{itemize}
	
	Higher-dimensional nearly Gorenstein rings are discussed in some situations, e.g., Hibi rings (\cite{HHS}), edge rings(\cite{HHS2}), Ehrhart rings of polytopes (\cite{HKMM, Mz}), and Stanley-Reisner rings (\cite{Ms}).
	Considering these results in conjunction with our current findings, we notice that nearly Gorenstein rings that are not Gorenstein are rare.
	This suggests that the nearly Gorenstein property is indeed close to the Gorenstein property.
\end{Remark}

\begin{Example}
The following are nearly Gorenstein rings with $\dim R_1 = \dim R_2 = 4$, $\dim R_3=3$, and $\dim R_4 = 2$ under the condition that the quotient ring of them by all $y_i$'s and $z_j$'s are isomorphic to numerical semigroup rings.

\begin{itemize}
	\item $R_1 = k[[X_1,X_2,X_3,Z_1,Z_2,Z_3]]/\detid_2\begin{pmatrix}X_2 & X_3 & X_1\\X_1^{\ell_1} + Z_1 & X_2^{\ell_2} + Z_2 & X_3^{\ell_3} + Z_3\end{pmatrix}$
	\item $R_2 = k[[X_1,X_2,X_3,Y_1,Z_2,Z_3]]/\detid_2\begin{pmatrix}X_2 & X_3 & X_1^{m_1} + Y_1\\X_1 & X_2^{\ell_2} + Z_2 & X_3^{\ell_3} + Z_3\end{pmatrix}$
	\item $R_3 = k[[X_1, X_2, X_3, X_4, Y_1, Z_3]]/\detid_2\begin{pmatrix}X_2 & X_3 & X_4 & X_1^{m_1} + Y_1\\X_1 & X_2 & X_3^{\ell_3} + Z_3 & X_4\end{pmatrix}$
	\item $R_4 = k[[X_1, X_2, \ldots , X_n, Y_1]] / \detid_2\begin{pmatrix}
		X_2 & \cdots & X_{n-1} & X_n & X_1^{m_1} + Y_1\\
		X_1 & \cdots & X_{n-2} & X_{n-1}^{\ell_{n-1}} & X_n^{\ell_n}
		\end{pmatrix}$
\end{itemize}
\end{Example}

\begin{Example}
	When $\dim R=1$, every almost Gorenstein local ring is nearly Gorenstein; however, in higher-dimensional cases, it does not hold.
	The theory discussed in this section also readily provides such examples.
	
	Let $n \ge 3$ be an integer and let $\ell_1, \ell_{n-1}, \ell_n \ge 1$ be positive integers such that $\ell_i \ge 2$ for some $i \in \{1, n-1, n\}$.
	Put
	$$R = k[[X_1,X_2,\ldots, X_n,Y_1, Y_2, \ldots , Y_n, Z_1]]/\detid_2\begin{pmatrix}X_2+Y_2 & \cdots & X_n + Y_n & X_1+Y_1\\X_1^{\ell_1} + Z_1 & \cdots  & X_{n-1}^{\ell_{n-1}} & X_n^{\ell_n}\end{pmatrix}.$$
	Then $R$ is of dimension $n+2$ and not a nearly Gorenstein local ring by Lemma \ref{lambda} (1).
	On the other hand, $R$ is an almost Gorenstein local ring by \cite[Theorem 7.8]{GTT}.
\end{Example}

\end{document}